\documentclass[a4paper,11pt]{article}
\usepackage[margin=2.5cm]{geometry}
\usepackage{graphicx}
\usepackage{amsmath}
\usepackage{amsfonts}
\usepackage{amssymb}
\usepackage{amsthm}
\usepackage{array}
\usepackage{float}
\usepackage{subcaption}
\usepackage{overpic}
\usepackage{pgfplots}
\pgfplotsset{compat=newest}
\usepackage{tikz}
\usepackage{bm}
\usepackage{todonotes}

\newtheorem{theorem}{Theorem}
\newtheorem{corollary}{Corollary}
\newtheorem{lemma}{Lemma}

\newtheorem{remark}{Remark}

\begin{document}

\title{On the optimization of the fixed-stress splitting for \\ Biot's equations}

\author{Erlend Storvik\thanks{Department of Mathematics, University of Bergen, Bergen, Norway; $\{$\texttt{erlend.storvik@uib.no, jakub.both@uib.no, jan.nordbotten@uib.no, florin.radu@uib.no}$\}$} 
\and 
Jakub W.\ Both\footnotemark[1]
\and
Kundan Kumar\thanks{Department of Mathematics and Computer Science, Karlstad University, Karlstad, Sweden; $\{$\texttt{kundan.kumar@kau.se}$\}$}
\and
Jan M.\ Nordbotten\footnotemark[1]\hspace{0.4em}$^,$\thanks{Department of Civil and Environmental Engineering, Princeton University, Princeton NJ, USA}
\and
Florin A.\ Radu\footnotemark[1]}



\maketitle

\begin{abstract}In this work we are interested in effectively solving the quasi-static, linear Biot model for poromechanics. We consider the fixed-stress splitting scheme, which is a popular method for iteratively solving Biot's equations. It is  well-known that the convergence of the method is strongly dependent on the applied stabilization/tuning parameter. In this work, we propose a new approach to optimize this parameter. We show theoretically that it depends also on the fluid flow properties and not only on the mechanics properties and the coupling coefficient. The type of analysis presented in this paper is not restricted to a particular spatial discretization. We only require it to be inf-sup stable. The convergence proof applies also to low-compressible or incompressible fluids and low-permeable porous media. Illustrative numerical examples, including random initial data, random boundary conditions or random source terms and a well-known benchmark problem, i.e. Mandel's problem are performed. The results are in good agreement with the theoretical findings. Furthermore, we show numerically that there is a connection between the inf-sup stability of discretizations and the performance of the fixed-stress splitting scheme.\end{abstract}

\section{Introduction}\label{sec:introduction}

There is currently a strong interest on numerical simulation of poromechanics, i.e.\ fully coupled porous media flow and mechanics. This is due to the high number of societal relevant applications of poromechanics, like geothermal energy extraction, life sciences or $\text{CO}_2$ storage to name a few. The most common used mathematical model for poromechanics is the quasi-static, linear Biot model, see e.g.\  \cite{coussy}:

Find $({\bf u}, p) $ such that
\begin{eqnarray}
-\nabla \cdot (2 \mu \varepsilon( \bm u) + \lambda \nabla\cdot {\bm u I})+\alpha \nabla p &=& {\bm f}, \label{eq:mechanics} \\ [1ex]
\frac{\partial}{\partial t}\left(\frac{p}{M} + \alpha \nabla \cdot {\bm u}\right) - \nabla \cdot (\kappa(\nabla p - {\bm g}\rho))& = &S_f, \label{eq:flow}
\end{eqnarray}
where ${\bm u}$ is the displacement, $\varepsilon({\bm u}) = \frac{1}{2}(\nabla{\bm u + \nabla \bm u}^\top)$ is the (linear) strain tensor, $\mu, \lambda$ are the Lam\'{e} parameters, $\alpha$ is the Biot-Willis constant, $p, \rho$ are fluid's pressure and density, respectively, $M$ is a compressibility constant, ${\bm g}$ the gravitational vector and $\kappa$ is the permeability. The source terms ${\bm f}$ and $S_f$ represent the density of applied body forces and a forced fluid extraction or injection process, respectively. 

There are plenty of works concerning the discretization of Biot's equations  \eqref{eq:mechanics}--\eqref{eq:flow}. The most common temporal discretization is based on backward Euler, see e.g.\ \cite{mikelicwang,jakubAML}. Many combinations of spatial discretizations have been proposed and analyzed. We mention cell-centered finite volumes \cite{jan}, continuous Galerkin for the mechanics and mixed finite elements for the flow \cite{phillips,yi2016,jakubAML,berger}, mixed finite elements for flow and mechanics \cite{yi2016}, non-conforming finite elements \cite{xu}, the MINI element \cite{carmenMINI}, continuous or discontinuous Galerkin \cite{riviere,riviereDG} or multiscale methods \cite{dana2,tchelepi,castelletto}. Continuous and discontinuous higher-order Galerkin space time elements were proposed in \cite{bause}. Adaptive computations were considered e.g.\ in \cite{ern}. A Monte-Carlo approach was proposed in \cite{fred}. For a discussion on the stability of the different spatial discretizations we refer to the recent papers \cite{haga2012,carmen}.

Independently of the chosen discretization there are two alternatives for solving Biot's equations: monolithically or by using an iterative splitting algorithm. The former has the advantage of being unconditionally stable, whereas a splitting method is much easier to implement, typically building on already available, separate numerical codes for porous media flow and for mechanics. However, a naive splitting of Biot's equations will lead to an unstable scheme \cite{Kim}. To overcome this, one adds a stabilization term in either the mechanics equation (the so-called {\it undrained split scheme}) or in the flow equation (the {\it fixed-stress splitting scheme}) \cite{settari1998}. The splitting methods have very good convergence properties, making them a valuable alternative to monolithic solvers for simulation of the linear Biot model, see e.g.\ \cite{settari1998,Kim,andro,jakubAML}. In the present work we will discuss the fixed-stress splitting scheme; a similar analysis can be performed for the undrained split scheme.

After applying backward Euler in time to \eqref{eq:mechanics}--\eqref{eq:flow} and discretizing in space (using finite elements or finite volumes), one has to solve a fully coupled, discrete system at each time step. The fixed-stress method is an iterative splitting scheme to solve this fully coupled system. If we denote by $i$ the iteration index, one looks to find a pair $({\bf u^i}, p^i)$ to converge to the solution $({\bf u}, p)$, when $i \rightarrow +\infty$. Algorithmically, one solves first the flow equation \eqref{eq:flow} using the displacement from the last iteration, then solves the mechanics equation \eqref{eq:mechanics} with the updated pressure and iterates until convergence is achieved. To ensure the convergence \cite{Kim,andro,jakubAML}, one needs to add a stabilization term $L(p^i - p^{i-1})$ to the flow equation \eqref{eq:flow}. The free to be chosen parameter $L \ge 0$ is called the stabilization or tuning parameter. Its choice is the deciding element for the success of the algorithm, because the number of iterations (and therefore the speed of the algorithm) strongly depends on the value of $L$, see \cite{bause,jakubAML,jakubuwe,mikelicwang,dana}. Moreover, a too small or too big $L$ will lead to no convergence.

The initial derivation of the fixed-stress scheme had a physical motivation \cite{settari1998,Kim}: one 'fixes the (volumetric) stress', i.e.\ imposes $K_{dr}{\nabla \cdot \bf u}^i - \alpha p^i = K_{dr} \nabla \cdot {\bf u}^{i-1} - \alpha p^{i-1}$ and uses this to replace $\alpha {\nabla \cdot \bf u}^i$ in the flow equation. Here, $K_{dr}$ is the physical, drained bulk modulus. The resulting stabilization parameter $L$, called from now on the {\it physical} parameter, is $L_{phys} = \dfrac{\alpha^2} {K_{dr}}$ (it depends on the mechanics and the coupling coefficient). Consequently, $L_{phys}$ was the recommended value for the stabilization parameter, and the general opinion was that the method is not converging (it is not stable) for $L < L_{phys}$. In 2013, a rigorous mathematical analysis of the fixed-stress scheme was for the first time performed in \cite{andro}, where the authors show that the scheme is a contraction for any stabilization parameter $L \ge  \dfrac{L_{phys}} {2}$. This analysis was confirmed in \cite{jakubAML} for heterogeneous media, and by using a simpler technique. Noticeable, the same result was obtained also for both continuous or discontinuous Galerkin, higher order space-time elements in \cite{bause,bause2018}, implying that the values of the tuning parameter are not depending on the order of the used elements. A legitimate question arises immediately: is now $L_{phys}$ or  $\dfrac{L_{phys}} {2}$ the optimal stabilization parameter (optimal in the sense that the convergence of the scheme is fastest, i.e.\ the number of iteration is smallest)? The question is relevant, because, as mentioned already above, the number of iterations can differ considerably depending on the choice of the stabilization parameter \cite{bause,jakubAML,jakubuwe,mikelicwang} (unless one uses the fixed-stress scheme  as a preconditioner for a monolithic solver, as done in \cite{manuel, white} or as a smoother for a multigrid solver \cite{pacomultigrid}). The  aim of the present paper is to answer this open question.  

In a recent study \cite{jakubuwe}, the authors considered different numerical settings and looked at the convergence of the fixed-stress splitting scheme. They determined numerically the optimal stabilization parameter for each considered case. This study, together with the previous results presented in \cite{mikelicwang} and \cite{jakubAML} is suggesting that the optimal parameter is actually a value in the interval $\left[\dfrac{L_{phys}} {2}, L_{phys} \right]$, depending on the data. Especially, the optimal parameter depends on the boundary conditions and also on the flow parameters, not only on the mechanics and coupling coefficient. Nevertheless, to the best of our knowledge, there exists no theoretical evidence for this in the literature so far.

In this paper we show for the first time that the optimal stabilization parameter for the fixed-stress scheme is neither $\dfrac{L_{phys}} {2}$ nor $L_{phys}$, but depends also on the flow parameters. The values  $\dfrac{L_{phys}} {2}$, $L_{phys}$ are obtained as limit situations. We prove first that the fixed-stress scheme converges linearly and then derive a theoretical optimal parameter, by minimizing the rate of convergence. 
The proof techniques in \cite{jakubAML} are improved to reach the new results. For this we require the discretization to be inf-sup stable. Essentially, this allows us to control errors in the pressure by those in the stress. A consequence of our theoretical result is that the fixed-stress splitting scheme also converges in the limit case of low-compressible fluids and low-permeable porous media, which has been not proved before. Finally, we perform numerical computations to test the optimized parameter. As can be seen in Section~\ref{sec:numerics}, the numerical results are sustaining the theory. In particular, we remark the connection between inf-sup stability and the performance of the fixed-stress scheme: a not inf-sup stable discretization leads to non-monotonic behavior of the splitting scheme with respect to the problems parameters (e.g. the permeability).

To summarize, the main contributions of this work are:
\begin{itemize}
 \item an improved, theoretical convergence result for the fixed-stress splitting scheme under the assumption of an inf-sup stable discretization,

 \item the derivation of an optimized tuning parameter depending on both mechanics and fluid flow parameters, and
 
 \item the numerical evidence that not inf-sup stable discretizations lead to non-monotonic behavior of the fixed-stress scheme w.r.t. to data (e.g. the permeability).
 
\end{itemize}

We mention that the fixed-stress scheme can be also used for non-linear extensions of Biot's equations, see \cite{manuel} for non-linear water compressibility and \cite{jakub2018,jakubProcBiot,Mardal,Kraus} for unsaturated flow and mechanics. In these cases, one combines a linearization technique, e.g.\ the $L-$scheme \cite{list,pop} with the splitting algorithm. The convergence of the resulting scheme can be proved rigorously \cite{manuel,jakub2018}. The fixed-stress method has been as well applied in connection with fracture propagation \cite{scoti} and phase field models \cite{lee}. Finally, we would like to mention some valuable variants of the fixed-stress method: the multirate fixed-stress method \cite{kundan}, the multiscale fixed-stress method \cite{dana} and the parallel-in-time fixed-stress method \cite{paco2018}.

The paper is structured as follows. The notations, discretization and the fixed-stress scheme are presented in Sec. \ref{sec:discretization}. The analysis of the convergence and the optimization are the subject of Sec. \ref{sec:analysis}. In Sec. \ref{sec:numerics} a numerical example is presented. Finally, the conclusions are given in Sec. \ref{sec:conclusions}.


\section{The numerical scheme for solving Biot's model} \label{sec:discretization}

In this paper we use common notations in functional analysis. We denote by $\Omega \subset \mathbb{R}^d$ a Lipschitz domain, $d$ being the spatial dimension. Further, we make use of the spaces $L^2(\Omega)$ and $H^1(\Omega)$, where $L^2(\Omega)$ is the Hilbert space of Lebesgue measurable, square integrable functions on $\Omega$ and $H^1(\Omega)$ is the Hilbert space of functions in $L^2(\Omega)$ with derivatives (in the weak sense) in $L^2(\Omega)$. We denote by $\langle \cdot, \cdot \rangle$ and $\| \cdot \|$ the inner product and the associated norm of $L^2(\Omega)$ and by $\|\cdot\|_{H^1(\Omega)}$ the standard $H^1(\Omega)$-norm. Vectors and tensors are written bold, and sometimes the scalar product and the norm will be taken for vectors and tensors. $T$ will denote the final time.

We solve the Biot equations \eqref{eq:mechanics}--\eqref{eq:flow} on the domain $\Omega \times (0,T)$ together with homogeneous Dirichlet boundary conditions and some initial condition. This choice of boundary conditions is only for the ease of notation; all theoretical analysis can be done equivalently with arbitrary Dirichlet or Neumann boundary conditions. We discretize in time by using the backward Euler method. We consider an uniform grid, with the time step size $\tau := \dfrac{T}{N}, N \in \mathbb{N}$ and $t_n := n \tau, n \in \mathbb{N}$. Throughout this work, the index $n$ will refer to the time level.

For the spatial discretization we use a two-field formulation and introduce two generic discrete spaces ${\mathbf V}_h$ and $Q_h$, associated with displacements and pressures. The most prominent example is the Taylor-Hood element; P2-P1 for displacement and pressure. Nevertheless, the analysis below can be extended without difficulties to a three-field formulation as e.g.\ in \cite{berger,jakubAML,phillips} or other formulations.

In this way, the fully discrete, weak problem reads: \\[1ex]
Let $n\geq 1$ and assume $({\bm u}_h^{n-1}, p_h^{n-1}) \in {\mathbf V}_h\times Q_h$ are given. Find $({\bm u}_h^n, p_h^n) \in {\mathbf V}_h\times Q_h$ such that 
\begin{eqnarray}
2 \mu\langle\varepsilon(\bm u_h^n),\varepsilon(\bm v_h)\rangle + \lambda \langle \nabla\cdot {\bm u}_h^n, \nabla \cdot {\bm v}_h\rangle- \alpha\langle  p_h^n,\nabla \cdot {\bm v}_h\rangle  & = & \langle {\bm f}^n,{\bm v}_h\rangle, \label{eq:discrete:1}
\\
\frac{1}{M}\langle p_h^n - p_h^{n-1} ,q_h\rangle + \alpha\langle \nabla \cdot ({\bm u}_h^n - {\bm u}_h^{n-1}), q_h\rangle  \nonumber\\
+ \tau\langle \kappa \nabla p_h^n,\nabla q_h\rangle - \tau\langle \kappa{\bf g}\rho,\nabla q_h\rangle & = & \tau\langle S^n_f, q_h \rangle \label{eq:discrete:2}
\end{eqnarray}
for all ${\bm v}_h \in {\mathbf V}_h, q_h \in  Q_h$.
For $n=1$ the functions $({\bm u}_h^{n-1}, p_h^{n-1})$ are obtained by using the initial condition.

We can now introduce the fixed-stress splitting scheme. We denote by $i$ the iteration index and by $n\ge 1$ the time level. The {\bf fixed-stress splitting  scheme} reads \cite{settari1998,Kim,mikelicwang,jakubAML}:\\[1ex]
For $i \ge 1$, given $L\geq 0$, $({\bm u}_h^{n-1}, p_h^{n-1}), ({\bm u}_h^{n,i-1}, p_h^{n,i-1}) \in {\mathbf V}_h \times Q_h$ find $({\bm u}_h^{n,i}, p_h^{n,i}) \in {\mathbf V}_h\times Q_h$ such that 
\begin{eqnarray}
2 \mu\langle\varepsilon(\bm u_h^{n,i}), \varepsilon(\bm v_h)\rangle + \lambda \langle \nabla\cdot {\bm u}_h^{n,i}, \nabla \cdot {\bm v}_h\rangle - \alpha\langle  p_h^{n,i},\nabla \cdot {\bm v}_h\rangle & = &\langle {\bm f}^n,{\bm v}_h\rangle, \label{eq:fs:1}
\\
\frac{1}{M}\langle p_h^{n,i} - p_h^{n-1}, q_h\rangle + \alpha\langle \nabla \cdot ({\bm u}_h^{n,i-1} - {\bm u}_h^{n-1}),q_h\rangle \nonumber \\
+ L\langle p_h^{n,i}-p_h^{n,i-1},q_h\rangle  + \tau\langle \kappa\nabla p_h^{n,i},\nabla q_h\rangle - \tau\langle \kappa{\bf g}\rho,\nabla q_h\rangle & = &\tau\langle S^n_f, q_h \rangle \label{eq:fs:2}
\end{eqnarray}
for all ${\bm v}_h \in {\mathbf V}_h, q_h \in  Q_h$. We start the iterations with the solution at the last time step, i.e. $({\bm u}_h^{n,0}, p_h^{n,0}) := ({\bm u}_h^{n-1}, p_h^{n-1})$. We emphasize; the mechanics and flow problems decouple, allowing the use of separate simulators for both subproblems.

\section{Convergence analysis and optimization}\label{sec:analysis}
In this section we analyze the convergence of the scheme  \eqref{eq:fs:1}--\eqref{eq:fs:2}. We are in particular interested in finding an {\it optimal} stabilization parameter $L$. Before we proceed with the main result we need some preliminaries. 

Let $K_{dr}>0$ be defined as the coercivity constant 
\begin{equation}\label{Kdr}
2 \mu \|\varepsilon(\bm u)\|^2+\lambda \|\nabla \cdot{\bm u} \|^2 \geq K_{dr}\|\nabla \cdot{\bm u} \|^2\qquad\text{for all }\bm{u} \in {\mathbf V}_h.
\end{equation}
One can easily prove that the inequality above holds for the physical, drained bulk modulus $K_{dr}=\frac{2\mu}{d}+\lambda$, where $d$ is the spatial dimension. In practice, for effectively lower-dimensional situations, e.g.\ one-dimensional compression, $d$ can be chosen smaller than the spatial dimension, as~\eqref{Kdr} is assumed to hold only for a relevant subset of displacements $\bm{u}$, cf.\ proof of Theorem~\ref{theorem1}. For a more detailed discussion on values for $K_{dr}$ we refer to \cite{jakubuwe} and Section~\ref{section:random-parabola}. Consistent with the literature, in the remaining article, despite the discrepancy between $K_{dr}$ and the physically well-defined bulk modulus, we continue calling $K_{dr}$ the drained bulk modulus independent of its value. 

Throughout this paper we make use of the following two assumptions:\\[1ex]
{\bf Assumption (A1)}. All the constants $\mu, \lambda, M, K_{dr}, \alpha, \kappa, \rho$ are strictly positive. The vector $\bm{g}$ is constant.\\[1ex]
{\bf Assumption (A2)}. The discretization $\bm{V}_h\times Q_h$ is inf-sup stable.
\begin{remark} The assumptions made above are valid in nearly all the relevant physical situations. Therefore, our analysis has a very wide application range.
 \end{remark}
 From the second assumption follows Lemma \ref{Lemma1} by applying Corollary 4.1.1 in \cite{brezzi}, which states:
\begin{corollary}\label{cor:brezzi}
Let $V$ and $Q$ be Hilbert spaces, and let $B$ be a linear continuous operator from $V$ to $Q'$. Denote by $B^t$ the transposed operator of $B$. Then, the following two statements are equivalent:
\begin{itemize}
\item $B^t$ is bounding: $\exists \gamma > 0$ such that $\|B^tq\|_{V'}\geq\gamma\|q\|_Q$ $\forall q \in Q$
\item $\exists L_B \in \mathcal{L}(Q', V) $ such that $B(L_B(\xi))=\xi$  $\forall \xi\in Q'$ with $\|L_b\|=\dfrac{1}{\gamma}.$
\end{itemize}
\end{corollary}
\begin{lemma}\label{Lemma1} Assume (A2). There exists $\beta >0$ such that for any $p_h \in Q_h$ there exists ${\bm u}_h \in {\bf V}\!_h $ satisfying
$\langle \nabla \cdot {\bm u_h} ,q_h \rangle = \langle p_h, q_h \rangle $ for all $q_h \in Q_h$ and 
\begin{equation}
2 \mu\|{\bm \varepsilon(\bm u_h })\|^2 + \lambda \| \nabla\cdot {\bm u_h}\|^2\leq \beta \|p_h\|^2. \label{beta}
\end{equation}
\end{lemma}
\begin{proof}
Consider \ Corollary \ref{cor:brezzi} and let the continuous linear function from ${\bf V}_h$ to $Q_h'$ be $B(\bm u_h)(q_h) = \langle \nabla \cdot \bm u_h , q_h \rangle$. The first statement is the characterization of an inf-sup stable discretization, with inf-sup constant $\gamma$. Considering the second statement we have the existence of a linear function $L_B \in \mathcal{L}(Q_h',\bm{V}_h)$ such that $B(L_B(\langle p_h,\cdot\rangle))= \langle p_h,\cdot \rangle$ for all $p_h \in Q_h$ with $\| L_B\| = 1/\gamma$. Hence $L_B$ is giving us for each $p_h \in Q_h$ the corresponding $\bm u_h \in {\bf V}_h$ such that 
$$ \langle \nabla \cdot \bm u_h, q_h \rangle = B(L_B(\langle p_h,\cdot\rangle ) )(q_h) = \langle p_h, q_h\rangle$$ for all $q_h \in Q_h$. Now the following chain of inequalities holds true,
$$2 \mu\|{\bm \varepsilon(\bm u_h })\|^2 + \lambda \| \nabla\cdot {\bm u_h}\|^2\leq C \|\bm u_h\|_{H^1(\Omega)}^2 \leq C\|L_B\|^2\|p_h\|^2$$ where the first one follows from Young's inequality with $C$ depending only on the Lam\'e parameters, and the second inequality results from the operator norm, $$\|L_B\|=\sup_{p_h\in Q_h,\ p_h\neq 0} \frac{\|L_B(\langle p_h, \cdot \rangle)\|_{H^1(\Omega)}}{\|\langle p_h, \cdot\rangle\|_{Q_h'}}= \sup_{\scriptsize \substack{p_h\in Q_h,p_h\neq 0 \\ \bm{u}_h = L_B(\langle p_h, \cdot\rangle)}} \frac{\|\bm u_h\|_{H^1(\Omega)}}{\|p_h\|_{Q_h}}.$$ Then we have our desired inequality, $$2 \mu\|{\bm \varepsilon(\bm u_h })\|^2 + \lambda \| \nabla\cdot {\bm u_h}\|^2\leq \frac{C}{\gamma^2}\|p_h\|^2=\beta\|p_h\|^2.$$

\end{proof}

\begin{remark} 
The constant $\beta$ above depends on $\mu$, $\lambda$, the domain $\Omega$ and on the choice of the finite dimensional spaces $ {\mathbf V}_h$ and $Q_h$. For more information see for example \cite{zsuppan}. 
\end{remark}

We can now give our main convergence result.
\begin{theorem}\label{theorem1}
Assume that (A1) and (A2) hold true and let $\delta\in(0,2]$. Define the iteration errors as ${\bm{e}^{n,i}_{\bm{u}} }:={\bm u}_h^{n,i}-{\bm u}_h^n$ and $e_p^{n,i}:=p_h^{n,i}-p_h^n$ where $\bm u_h^{n,i},p_h^{n,i}$ are solutions to \eqref{eq:fs:1}--\eqref{eq:fs:2} and $\bm u_h^n, p_h^n$ are solutions to \eqref{eq:discrete:1}--\eqref{eq:discrete:2}. The fixed-stress splitting scheme \eqref{eq:fs:1}--\eqref{eq:fs:2}  converges linearly for any $L \geq \frac{\alpha^2}{\delta K_{dr}}$, with a convergence rate given by
\begin{equation}\label{firstrate}
\mathrm{rate}(\delta)=\frac{L}{L+\frac{2}{M}+\frac{2\tau\kappa}{C_\Omega^2}+(2-\delta)\frac{\alpha^2}{\beta}},
\end{equation}
through the error inequalities
\begin{equation}\label{eq:presconvergence}
\|e_p^{n,i}\|^2\leq \mathrm{rate}(\delta)\|e_p^{n,i-1}\|^2
\end{equation}
\begin{equation}\label{eq:dispconvergence}
2\mu\|\varepsilon( \bm{e}_{\bm{u}}^{n,i})\|^2+\lambda\| \nabla\cdot \bm{e}_{\bm{u}}^{n,i}\|^2\leq \frac{\alpha^2}{K_{dr}}\|e_p^{n,i}\|^2
\end{equation}
where  $C_\Omega$ the Poincar\'e constant and $\beta$ the constant from \eqref{beta}.
\end{theorem}

\begin{proof}
Subtract \eqref{eq:fs:1}, \eqref{eq:fs:2} from \eqref{eq:discrete:1}, \eqref{eq:discrete:2}, respectively,  to obtain the error equations
\begin{equation}\label{bioterror}
\begin{cases}
(i) \ \ 2 \mu\langle \varepsilon( \bm{e}_{\bm{u}}^{n,i}), \varepsilon( \bm v_h)\rangle + \lambda \langle \nabla\cdot \bm{e}_{\bm{u}}^{n,i}, \nabla \cdot {\bm v}_h\rangle- \alpha\langle  e_p^{n,i},\nabla \cdot {\bm v}_h\rangle  = 0
\\
(ii) \ \ \frac{1}{M}\langle e_p^{n,i},q_h\rangle + \alpha\langle \nabla \cdot \bm{e}_{\bm{u}}^{n,i-1},q_h\rangle +L\langle e_p^{n,i}-e_p^{n,i-1},q_h\rangle + \tau \langle \kappa \nabla e_p^{n,i},\nabla q_h\rangle = 0.
\end{cases}
\end{equation}
To prove \eqref{eq:dispconvergence} test \eqref{bioterror}(i) with $\bm v_h=\bm{e}^{n,i}_{\bm{u}}$, and apply the Cauchy Schwarz inequality and Young's inequality to the pressure term to obtain
\begin{equation}
2\mu\|\varepsilon( \bm{e}_{\bm{u}}^{n,i})\|^2+\lambda\| \nabla\cdot \bm{e}_{\bm{u}}^{n,i}\|^2\leq \frac{\alpha^2}{2K_{dr}}\|e_p^{n,i}\|^2+\frac{K_{dr}}{2}\|\nabla\cdot \bm{e}_{\bm{u}}^{n,i}\|^2.
\end{equation}
We now get \eqref{eq:dispconvergence} by applying \eqref{Kdr}.

In order to prove  \eqref{eq:presconvergence} test \eqref{bioterror} with $q_h = e_p^{n,i}$ and $\bm v_h =\bm{e}_{\bm{u}}^{n,i}$, add the resulting equations and use the algebraic identity
$$\langle 	e_p^{n,i}-e_p^{n,i-1},e_p^{n,i}\rangle= \frac{1}{2}\left(\| e_p^{n,i}-e_p^{n,i-1}\|^2 + \|e_p^{n,i}\|^2 - \|e_p^{n,i-1}\|^2\right)$$
 to get 
\begin{align*}
& 2 \mu\| \varepsilon( \bm{e}_{\bm{u}}^{n,i})\|^2 + \lambda \| \nabla\cdot \bm{e}_{\bm{u}}^{n,i}\|^2+ \frac{1}{M}\| e_p^{n,i}\|^2 - \alpha \langle  e_p^{n,i},\nabla \cdot (\bm{e}_{\bm{u}}^{n,i} - \bm{e}_{\bm{u}}^{n,i-1}) \rangle \\&\quad+ \tau\kappa\|\nabla e_p^{n,i}\|^2 + \frac{L}{2}\|e_p^{n,i}-e_p^{n,i-1}\|^2 + \frac{L}{2}\|e_p^{n,i}\|^2 = \frac{L}{2}\|e_p^{n,i-1}\|^2.
\end{align*}
Using now equation \eqref{bioterror}($i$) tested with $\bm v_h=\bm{e}_{\bm{u}}^{n,i}-\bm{e}_{\bm{u}}^{n,i-1}$  in the above yields
\begin{align}\label{rearrangetoget}
& 2 \mu\|\varepsilon(\bm{e}_{\bm{u}}^{n,i})\|^2 + \lambda \| \nabla\cdot \bm{e}_{\bm{u}}^{n,i}\|^2+ \frac{1}{M}\| e_p^{n,i}\|^2 + \tau \kappa \| \nabla e_p^{n,i}\|^2 + \frac{L}{2}\|e_p^{n,i}\|^2 \nonumber
\\&\quad  + \frac{L}{2}\| e_p^{n,i}-e_p^{n,i-1}\|^2 = \frac{L}{2}\|e_p^{n,i-1}\|^2 + 2 \mu\langle\varepsilon( \bm{e}_{\bm{u}}^{n,i}),\varepsilon( \bm{e}_{\bm{u}}^{n,i}-\bm{e}_{\bm{u}}^{n,i-1})\rangle 
\\ & \nonumber \quad \hspace*{4cm}+ \lambda \langle \nabla\cdot \bm{e}_{\bm{u}}^{n,i}, \nabla \cdot (\bm{e}_{\bm{u}}^{n,i}-\bm{e}_{\bm{u}}^{n,i-1})\rangle.
\end{align}
By applying Young's inequality in \eqref{rearrangetoget} we obtain that for any $\delta> 0$ there holds
\begin{align}\label{postyoung}
& 2 \mu\|\varepsilon( \bm{e}_{\bm{u}}^{n,i})\|^2 + \lambda \| \nabla\cdot \bm{e}_{\bm{u}}^{n,i}\|^2+ \frac{1}{M}\| e_p^{n,i}\|^2 + \tau \kappa \|\nabla e_p^{n,i}\|^2 + \frac{L}{2}\|e_p^{n,i}\|^2 \nonumber
\\&\quad  + \frac{L}{2}\| e_p^{n,i}-e_p^{n,i-1}\|^2 = \frac{L}{2}\|e_p^{n,i-1}\|^2 +\frac{\delta}{2}(2\mu\|\varepsilon(\bm{e}_{\bm{u}}^{n,i})\|^2+\lambda\|\nabla\cdot \bm{e}_{\bm{u}}^{n,i}\|^2)
\\&\quad \quad\nonumber +\frac{1}{2\delta}(2\mu\|\varepsilon( \bm{e}_{\bm{u}}^{n,i}-\bm{e}_{\bm{u}}^{n,i-1})\|^2+\lambda\|\nabla \cdot (\bm{e}_{\bm{u}}^{n,i}-\bm{e}_{\bm{u}}^{n,i-1})\|^2).
\end{align}
To take care of the last term in \eqref{postyoung} consider equation \eqref{bioterror}($i$). Subtract iteration $i-1$ from iteration $i$, let ${\bm v_h}=\bm{e}_{\bm{u}}^{n,i}-\bm{e}_{\bm{u}}^{n,i-1}$ in the result  and apply the Cauchy-Schwarz inequality to get
\begin{align}\label{postcauchy}
2\mu \|\varepsilon( \bm{e}_{\bm{u}}^{n,i})-\varepsilon(\bm{e}_{\bm u}^{n,i-1})\|^2+\lambda\|\nabla\cdot (\bm{e}_{\bm{u}}^{n,i}-\bm{e}_{\bm{u}}^{n,i-1})\|^2
\leq \alpha \|e_p^{n,i}- e_p^{n,i-1}\|\|\nabla \cdot (\bm{e}_{\bm{u}}^{n,i}-\bm{e}_{\bm{u}}^{n,i-1})\|.
\end{align}
By using now \eqref{Kdr}, \eqref{postcauchy} implies
\begin{equation}\label{postKdr}
K_{dr}\|\nabla \cdot(\bm{e}_{\bm{u}}^{n,i}-\bm{e}_{\bm{u}}^{n,i-1})\| \leq \alpha \|e_p^{n,i}- e_p^{n,i-1}\|.
\end{equation}
Inserting~\eqref{postKdr} into \eqref{postcauchy}, yields
\begin{equation}\label{postcombined}
2\mu \|\varepsilon( \bm{e}_{\bm{u}}^{n,i})-\varepsilon(\bm{e}_{\bm u}^{n,i-1})\|^2+\lambda\|\nabla\cdot (\bm{e}_{\bm{u}}^{n,i}-\bm{e}_{\bm{u}}^{n,i-1})\|^2 \leq \frac{\alpha^2}{K_{dr}} \|e_p^{n,i}- e_p^{n,i-1}\|^2.
\end{equation}
By rearranging terms and inserting~\eqref{postcombined} into~\eqref{postyoung}, we immediately get
\begin{align*}
& \left(1-\frac{\delta}{2}\right)(2 \mu\|\varepsilon(\bm{e}_{\bm{u}}^{n,i})\|^2 + \lambda \| \nabla\cdot \bm{e}_{\bm{u}}^{n,i}\|^2)+ \frac{1}{M}\| e_p^{n,i}\|^2 + \tau \kappa \| \nabla e_p^{n,i}\|^2 + \frac{L}{2}\|e_p^{n,i}\|^2 
\\&\quad+ \frac{L}{2}\| e_p^{n,i}-e_p^{n,i-1}\|^2  \leq \frac{L}{2}\|e_p^{n,i-1}\|^2  +\dfrac{\alpha^2}{2\delta K_{dr}} \|e_p^{n,i}- e_p^{n,i-1}\|^2.
\end{align*}
Using that $L\geq \dfrac{\alpha^2}{\delta K_{dr}}$ and Poincar\'e's inequality we obtain from the above 
\begin{equation}\label{firstconvergence}
\left(1-\frac{\delta}{2}\right)(2 \mu\|\varepsilon(\bm{e}_{\bm{u}}^{n,i})\|^2 + \lambda \| \nabla\cdot \bm{e}_{\bm{u}}^{n,i}\|^2)+ \left(\frac{1}{M}+\frac{L}{2}+\frac{\tau\kappa}{C^2_\Omega}\right)\| e_p^{n,i}\|^2 \leq\frac{L}{2}\|e_p^{n,i-1}\|^2. 
\end{equation}
The result \eqref{firstconvergence} already implies that we have convergence of the scheme. 
In previous works, especially \cite{jakubAML} the conclusion at this point was that $L=\dfrac{\alpha^2}{2K_{dr}}$ is the optimal parameter. However, this does not consider the influence of the first term in \eqref{firstconvergence}. 
By Lemma \ref{Lemma1} we get that there exists ${\bm v_h \in V}\!_h$ such that $e_p^{n,i}=\nabla \cdot {\bm v_h}$ in a weak sense and 
\begin{equation}\label{postlemma2}
2 \mu\|\varepsilon( \bm v_h)\|^2 + \lambda \| \nabla\cdot {\bm v_h}\|^2\leq \beta \|e_p^{n,i}\|^2.
\end{equation}
By testing now \eqref{bioterror}($i$) with this ${\bm v}_h$, we get 
\begin{equation}\label{eq:5}
\alpha\|e_p^{n,i}\|^2  = 2 \mu\langle\varepsilon(\bm{e}_{\bm{u}}^{n,i}),\varepsilon( {\bm v}_h )\rangle + \lambda \langle \nabla\cdot \bm{e}_{\bm{u}}^{n,i}, \nabla \cdot {\bm v}_h\rangle.
\end{equation}
From \eqref{postlemma2} and \eqref{eq:5} and the Cauchy-Schwarz inequality we immediately obtain
\begin{equation}
\frac{\alpha^2}{\beta}\|e_p^{n,i}\|^2 \leq 2 \mu \|\varepsilon( \bm{e}_{\bm{u}}^{n,i})\|^2+ \lambda \|\nabla\cdot \bm{e}_{\bm{u}}^{n,i}\|^2,
\end{equation}
which together with  \eqref{firstconvergence} implies
\begin{equation}\label{secondconvergence}
\left(\frac{1}{M}+\frac{L}{2}+\frac{\tau\kappa}{C^2_\Omega}+\left(1-\frac{\delta}{2}\right) \frac{\alpha^2}{\beta}\right)\| e_p^{n,i}\|^2 \leq\frac{L}{2}\|e_p^{n,i-1}\|^2. \nonumber
\end{equation}
This gives the rate of convergence,  for $\delta\in(0,2]$ and $L\geq \dfrac{\alpha^2}{\delta K_{dr}}$
\begin{equation} 
\mathrm{rate}(\delta)=\frac{L}{L+\frac{2}{M}+\frac{2\tau\kappa}{C^2_\Omega}+(2-\delta)\frac{\alpha^2}{\beta}}.\nonumber
\end{equation}
\end{proof}
\begin{remark} One can easily extend the result for a heterogeneous media, i.e.\  $\kappa = \kappa( {\bf x} )$ as long as $\kappa$ is bounded from below by $\kappa_m > 0$. Also any of the other parameters can be chosen spatially dependent as long as they are bounded from below by appropriate constants larger than zero.
\end{remark}
\subsection{Optimality}\label{sec:optimality}
Let us now look at the rate obtained in \eqref{firstrate}. It is clear that the best choice of $L$ is $L = \dfrac{\alpha^2}{\delta K_{dr}}$, giving the rate
\begin{equation}\label{rate2}
\mathrm{rate}(\delta)=\frac{\frac{\alpha^2}{K_{dr}}}{\frac{\alpha^2}{ K_{dr}}+\delta(\frac{2}{M}+\frac{2\tau\kappa}{C_\Omega^2}+(2-\delta)\frac{\alpha^2}{\beta})},
\end{equation}
where $\delta \in (0,2]$ is still free to be chosen.
Minimizing \eqref{rate2} corresponds to maximizing 
$$\delta\left(\frac{2}{M}+\frac{2\tau\kappa}{C_\Omega^2}+(2-\delta)\frac{\alpha^2}{\beta}\right).$$
Let $A:=\frac{2}{M}+\frac{2\tau\kappa}{C_\Omega^2}+2\frac{\alpha^2}{\beta}$ and $B:=\frac{\alpha^2}{\beta}$ and we can see that the maxima of $\delta(A-\delta B)$ is $\delta=\frac{A}{2B}$. Therefore the best choice of $\delta$ is
\begin{equation}
\delta = \min\left\{\frac{A}{2B}, 2\right\} \in [1, 2],
\end{equation}
because $A \geq 2B$. This implies, the optimal choice for L is 
\begin{equation}
\label{optimal-tuning-parameter}
L = \dfrac{\alpha^2}{K_{dr} \,  \min\left\{\dfrac{A}{2B}, 2 \right\}}  \in \left[\frac{L_{phys}}{2}, L_{phys}\right].
\end{equation}
We especially remark the two extreme situations:
\begin{enumerate}
 \item[1)] $M$ small, $\kappa \tau$ big $\Rightarrow  L = \dfrac{L_{phys}}{2}$ and
 \item[2)] $M$ big, $\kappa \tau$ small $\Rightarrow  L = L_{phys}$.
\end{enumerate}
This implies that one should choose $L = \dfrac{L_{phys}}{2}$ for e.g. highly-compressible fluids or for highly-permeable media or for very large time steps. In contrast, one should take the physical parameter $ L = L_{phys}$ for e.g. low-compressible and low-permeable porous media or very small time steps. These theoretical results will be verified by numerical experiments in the following section.

\begin{remark}[Consequence for low-compressible fluids and low-permeable porous media]\label{remark:convergence-inf-sup}
 Previous convergence results in the literature for the fixed-stress splitting scheme have not predicted or guaranteed any convergence in the limit case $M\rightarrow\infty$ and $\kappa\rightarrow 0$ (by fixed time step size $\tau$). However, by Theorem~\ref{theorem1}, for inf-sup stable discretizations, convergence of the fixed-stress splitting scheme is guaranteed, even in the limit case.
\end{remark}

\section{Numerical examples}\label{sec:numerics}
In this section we verify numerically the theoretical results of Theorem \ref{theorem1}. In particular we show that for constant material properties, the practical optimal value of $\delta$ increases for increasing permeability, $\kappa$, as the theory predicts.  We also emphasize that this does not hold for inf-sup unstable discretizations, e.g. P1-P1.

Three test cases are considered:
\begin{enumerate}
\item An experiment in the unit square domain with source terms giving parabolas as analytical solution to the continuous problem, \eqref{eq:mechanics}--\eqref{eq:flow}, for homogeneous Dirichlet boundary conditions.
\item An L-shaped domain with source terms from test case $1$.
\item Mandel's problem.
\end{enumerate}
For the first test case, additionally, we perform a deeper investigation on the robustness of the optimal tuning parameter with respect to external influences as initial guesses, boundary conditions etc.

We are using a MATLAB code for solving the problem in a two field formulation both in a P2-P1 stable discretization and in a P1-P1 not inf-sup stable discretization. The results have been verified with the DUNE \cite{dune} based code used in \cite{jakubAML}. 

In all the plots we consider several permeabilities, $\kappa$. For each of them we solve \eqref{eq:fs:1}--\eqref{eq:fs:2} with a range of stabilization parameters $L=\frac{\alpha^2}{\delta K_{dr}}$. This is visualized through plots showing total numbers of iterations in the y-axis and $\delta$ in the x-axis. The domain of $\delta$ is varying slightly over the different test cases, but always contains the interval $(1,2]$ which the theory predicts to contain the optimal value through subsection \ref{sec:optimality}. The stars in each plot denote the theoretically calculated optimal value of $\delta$. 

As stopping criterion we apply the relative errors in infinity-norm, $\frac{\|u_h^{n,i}-u_h^{n,i-1}\|_\infty}{\|u_h^{n,i}\|_\infty}<\epsilon_{u,r}$ and $\tfrac{\|p_h^{n,i}-p_h^{n,i-1}\|_\infty}{\|p_h^{n,i}\|_\infty}<\epsilon_{p,r}$ where $\epsilon_{u,r}$ and $\epsilon_{p,r}$ are defined separately for the different test cases. 

\begin{remark}[Choice of $K_{dr}$]\label{remark:test-case-dependent-kdr}
If one knows the drained bulk modulus, $K_{dr}$, choosing the optimal stabilization parameter should be possible. However, as already mentioned in Section~\ref{sec:analysis}, this is problem dependent; finding the correct one might not be trivial. For our computations, we choose $K_{dr}$ so that the theoretical optimal stabilization parameter is actually the practical optimal one for the smallest considered permeability. We experience that it also fits quite nicely for the remaining permeabilities for that particular setup. For all problems we set $\beta =K_{dr}$. However, we stress that this actually is not a realistic choice of $\beta$, which in reality is larger than $K_{dr}$.
\end{remark}

\subsection{Unit square domain}\label{section:numerical-results:parabola}

In this test case we consider two setups on a unit square domain. For the first setup we apply homogeneous Dirichlet boundary conditions and zero initial data for both displacement and pressure. We employ source terms corresponding to the analytical solution of the continuous problem
$$u_1(x,y,t)=u_2(x,y,t)=\tfrac{1}{p_\mathrm{ref}} p(x,y,t)=txy(1-x)(1-y),\qquad (x,y)\in(0,1)^2,\ t\in(0,0.1),$$
regardless of permeability, Lam\'e parameters and the Biot-Willis constant, see Table~\ref{tab:CoefficientsCase1}. The pressure, $p$, is scaled by $p_\mathrm{ref}=10^{11}$ in order to balance  the orders of magnitude of the mechanical and fluid stresses for the chosen physical parameters. In the second setup we keep the initial data and source terms from the first setup while assigning homogeneous Dirichlet boundary conditions for the displacement everywhere but at the top, $\Gamma_N=(0,1)\times\{1\}$, where homogeneous natural boundary conditions are applied. For the pressure homogeneous boundary conditions are applied on the entire boundary.  For both setups, we compute one single time step from $0$ to $0.1$, and discretize the domain using a regular triangular mesh with mesh size $h=1/8$. Numerical tests have showed that multiple time steps and different mesh diameters yield similar performance results. The tolerances $\epsilon_{u,r}$ and $\epsilon_{p,r}$ are set to $10^{-12}$. Solutions for both setups are plotted for $\kappa=10^{-10}$ in Figure \ref{fig:SolTestCase1}.
To summarize, we have

\begin{itemize}
\item Setup 1: Homogeneous Dirichlet data on the entire boundary for displacement and pressure.
\item Setup 2: Homogeneous Dirichlet data for the pressure. Homogeneous Neumann data on top in the mechanics equation, homogeneous Dirichlet data everywhere else for the displacement.
\end{itemize}

The drained bulk modulus is set to $K_{dr}=1.6\mu+\lambda$ for setup 1 and $K_{dr} = 1.1\mu+\lambda$ for setup 2.

\begin{table}[h]
\begin{center}
\begin{tabular}{| l | c | r |}
\hline
Symbol & Name & Value\\
\hline
$\lambda$ & Lam\'e parameter 1 & $27.778\cdot 10^9$  \\
$\mu$ & Lam\'e parameter 2 &$41.667\cdot 10^9 $ \\
$\kappa$ & Permeability & $10^{-15}, 10^{-14}, ..., 10^{-10} $\\
M & Compressibility coefficient & $10^{11}$ \\
$\alpha$ & Biot-Willis constant & $1$ \\
$u_0$, $p_0$ & Initial data & 0\\
$h$ & Mesh diameter & $\frac{1}{8}$\\
$\tau$ & Time step size &  $0.1$ \\
$t_0$ & Initial time & $0$\\
$T$ & Final time & $0.1$\\
$\epsilon_{u,r}$ and $\epsilon_{p,r}$ & Tolerances & $10^{-12}$\\

\hline
\end{tabular}
\caption{Coefficients for test cases 1 and 2}
\label{tab:CoefficientsCase1}
\end{center}
\end{table}
\begin{figure}[h!]
\begin{subfigure}[b]{0.49\textwidth}
\includegraphics[width=\textwidth]{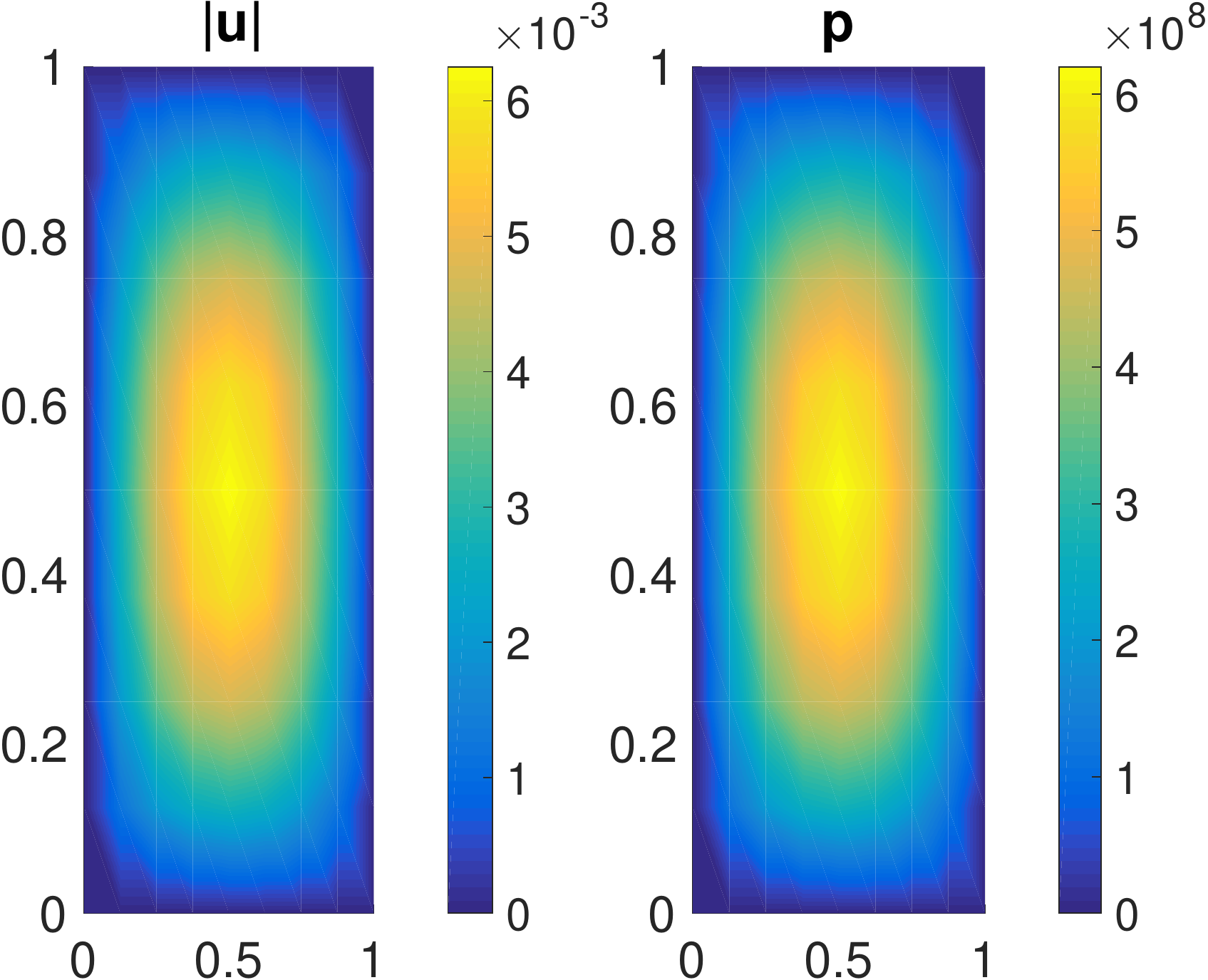}
\caption{Setup 1}
\label{fig:SolSetup1}
\end{subfigure}
\begin{subfigure}[b]{0.49\textwidth}
\includegraphics[width=\textwidth]{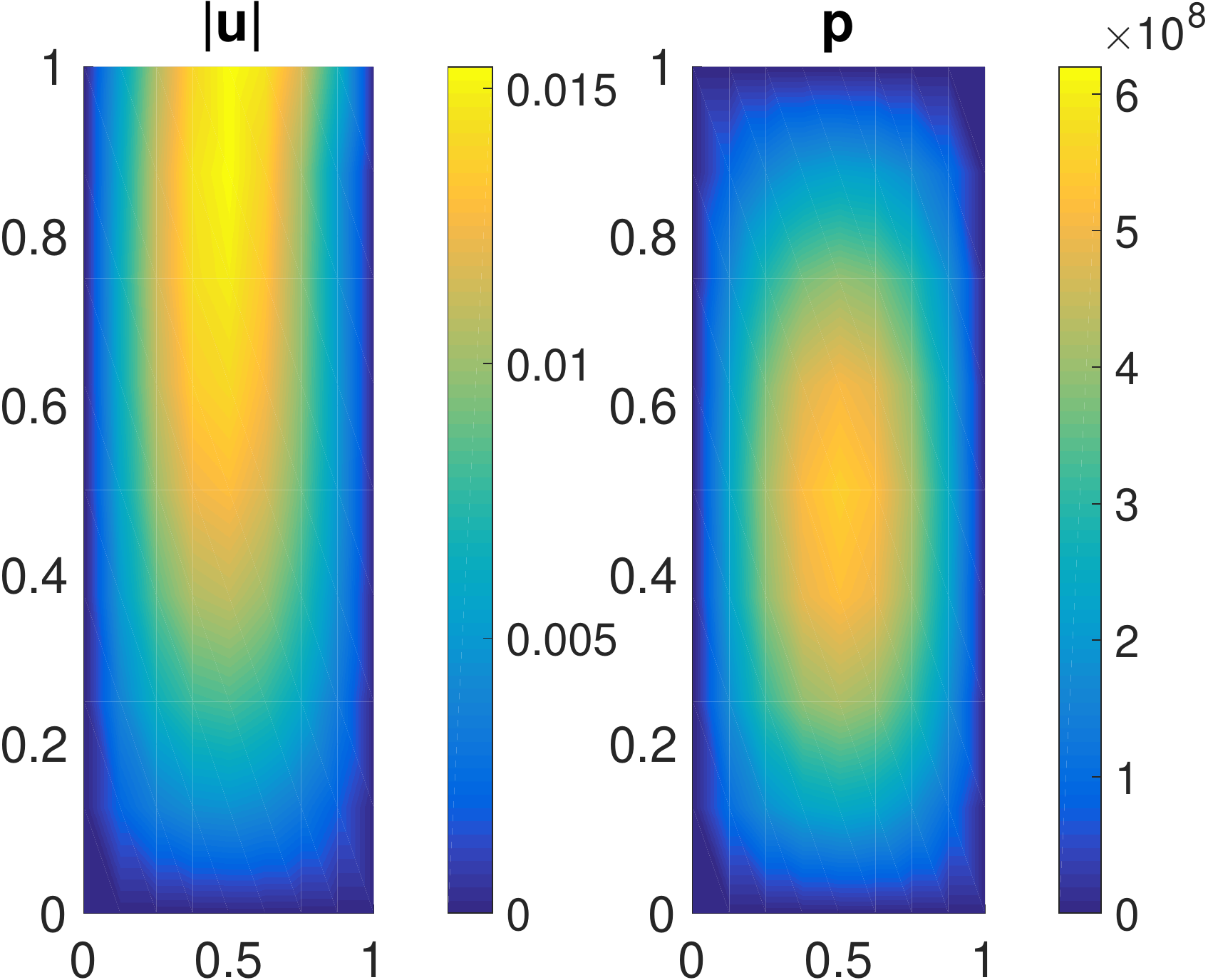}
\caption{Setup 2}
\label{fig:SolSetup2}
\end{subfigure}
\caption{Displacement (Left) and Pressure (Right) for test case 1 with $\kappa =10^{-10}$. Remark that the colors are scaled differently for the two displacements.}
\label{fig:SolTestCase1}
\end{figure}
\begin{figure}[h!]
\begin{subfigure}[b]{0.5\textwidth}
\includegraphics[width=\textwidth]{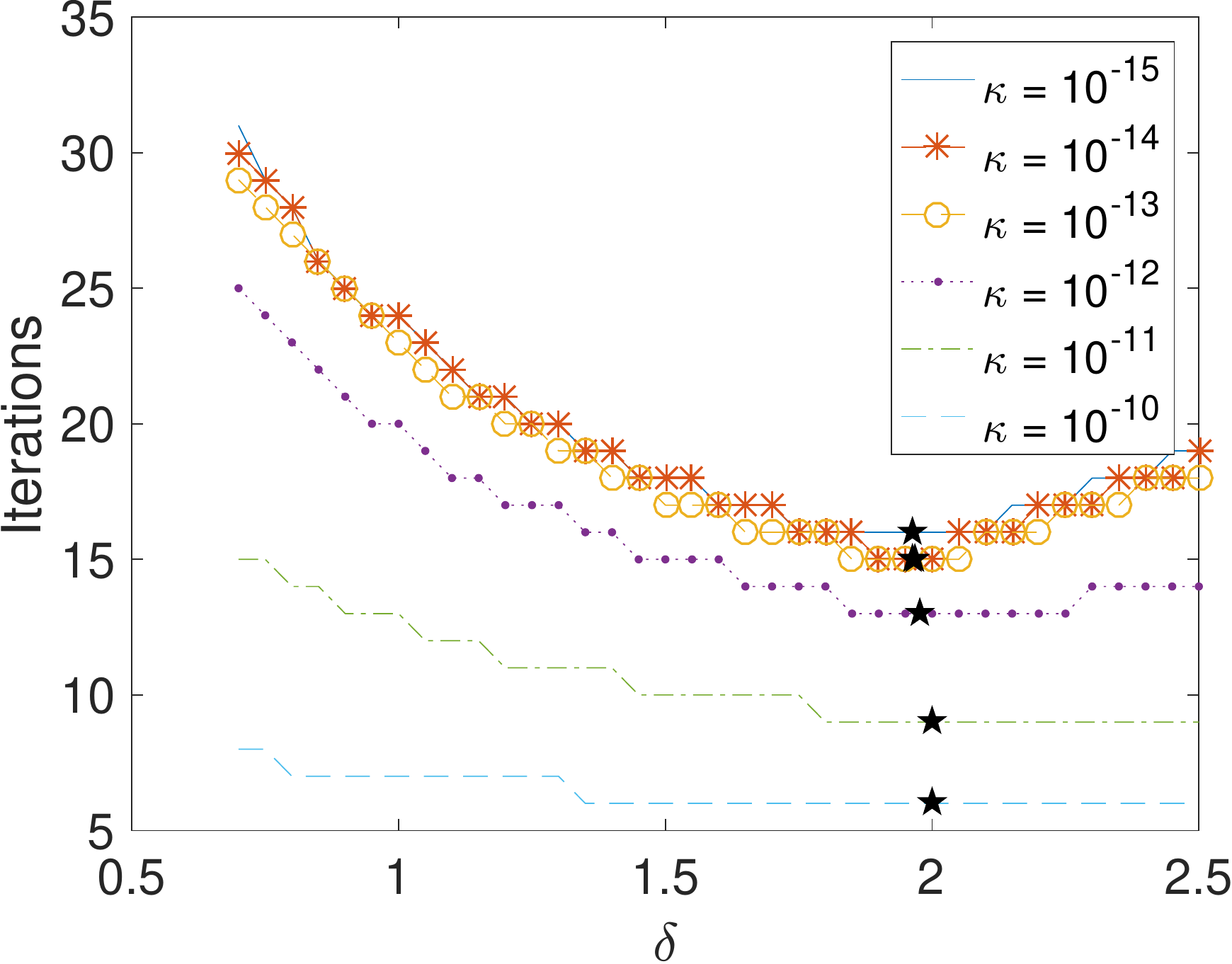}
\caption{P2-P1 discretization}
\label{fig:Setup1P2}
\end{subfigure}
\begin{subfigure}[b]{0.5\textwidth}
\includegraphics[width=\textwidth]{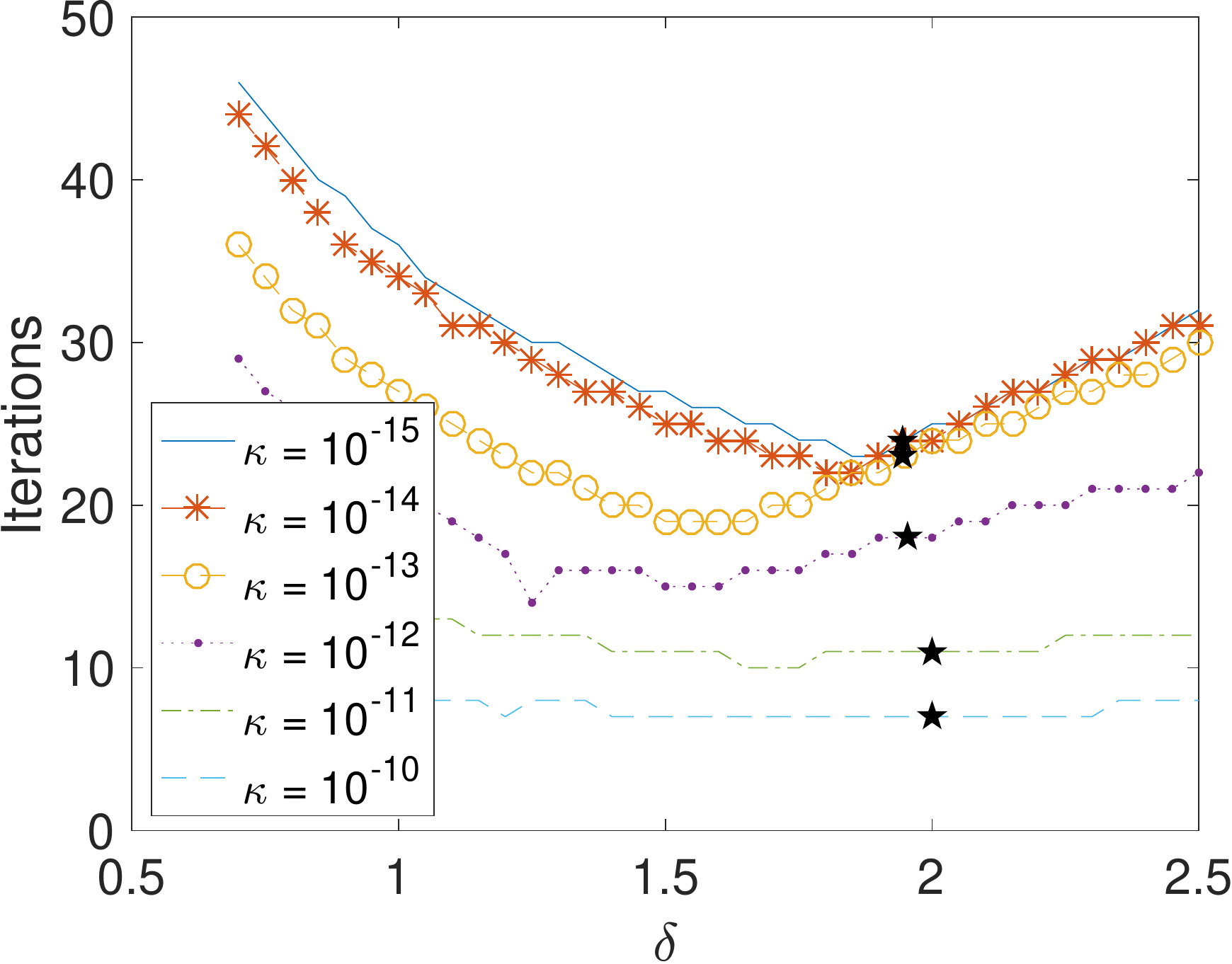}
\caption{P1-P1 discretization}
\label{fig:Setup1P1}
\end{subfigure}
\label{fig:setup1}
\caption{Test case 1 (Unit square domain): Setup 1, 
total iteration count for one time step applying stabilization parameter $L=\frac{\alpha^2}{\delta K_{dr}}$ with $K_{dr} = 
1.6\mu+\lambda$. The star represents the theoretically calculated optimal $\delta$. }
\end{figure}
\begin{figure}[h!]
\begin{subfigure}[b]{0.49\textwidth}
\includegraphics[width=\textwidth]{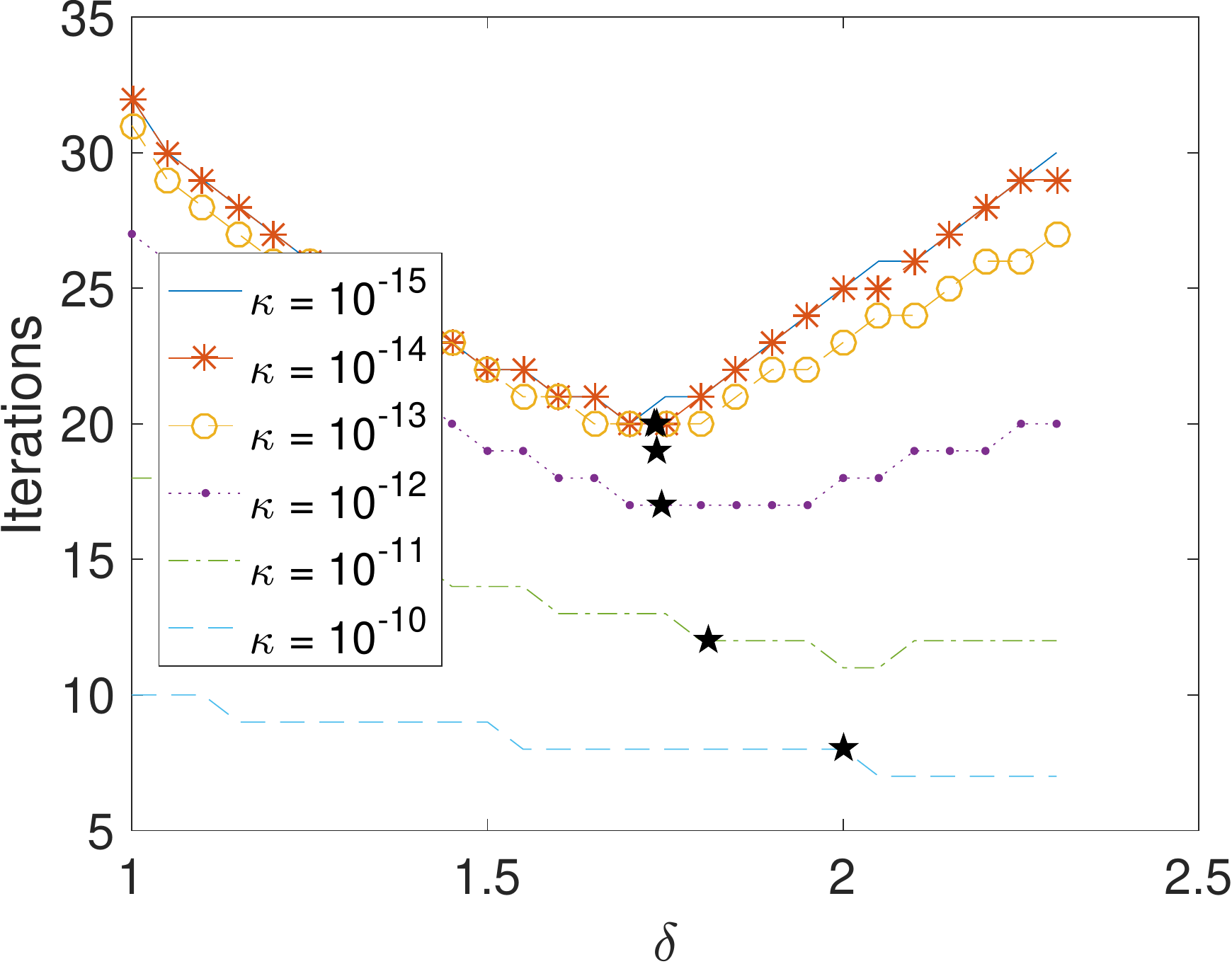}
\caption{P2-P1 discretization}
\label{fig:Setup2P2}
\end{subfigure}
\begin{subfigure}[b]{0.49\textwidth}
\includegraphics[width=\textwidth]{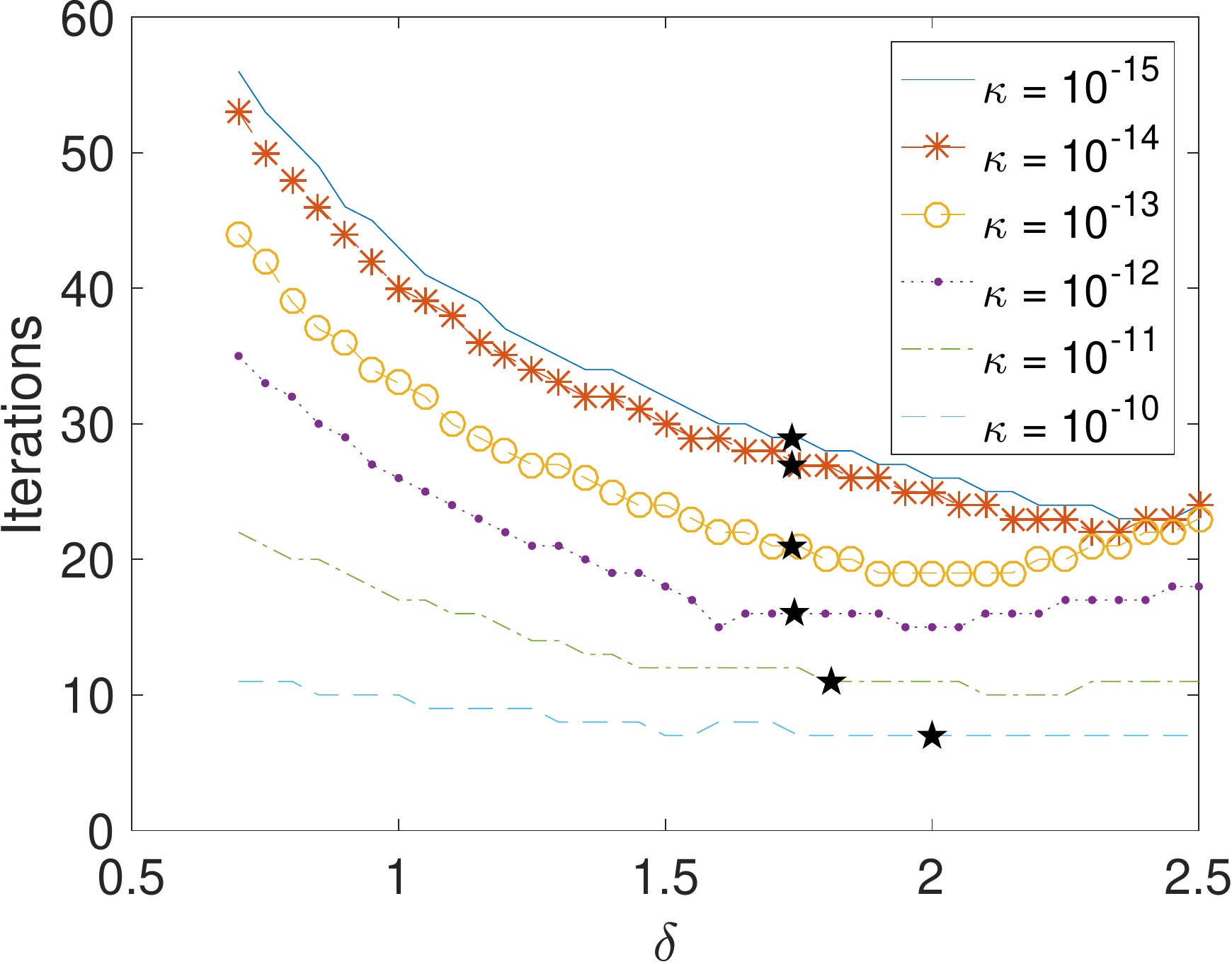}
\caption{P1-P1 discretization}
\label{fig:Setup2P1}
\end{subfigure}
\label{fig:setup2}
\caption{Test case 1 (Unit square domain): Setup 2, total iteration count for one time step applying stabilization parameter $L=\frac{\alpha^2}{\delta K_{dr}}$ with $K_{dr}=1.1\mu+\lambda$. The star represents the theoretically calculated optimal $\delta$.}
\end{figure}

We experience in the inf-sup stable discretizations, Figure \ref{fig:Setup1P2} and \ref{fig:Setup2P2}, that as $\kappa$ increases so does the optimal $\delta$ which is in accordance with Theorem \ref{theorem1}. However, when we have a not inf-sup stable discretization, Figure \ref{fig:Setup1P1} and \ref{fig:Setup2P1}, the behavior does not follow the same trend. In particular, we observe that for the three smallest permeability values, $\kappa = 10^{-15}$, $  \kappa=10^{-14}$ and $\kappa = 10^{-13}$, the optimal stabilization parameter is moving in the opposite direction to the theoretically calculated one as the stability of the discretization is lost.

\subsubsection{Robustness of the optimality of the tuning parameter}\label{section:random-parabola}

As already mentioned in remark~\ref{remark:test-case-dependent-kdr}, the
value for $K_{dr}$ depends on the particular test case; this directly
translates to the optimal choice of the tuning parameter
$L=\tfrac{\alpha^2}{\delta K_{dr}}$. We have experience this in
section~\ref{section:numerical-results:parabola}. Just changing the
distribution of the Dirichlet and Neumann boundaries has resulted in
quite different choices for $K_{dr}$. In the following, we analyze the
robustness of the optimal tuning parameter with respect to varying
numerical and physical data. For this, we revisit the test case from
section~\ref{section:numerical-results:parabola} and limit the
discussion to setup 1 and the P2-P1 discretization. We keep the setting
as before except for modifying single components at a time. For this, we
do not take fixed but random (uniformly distributed values of order
$\mathcal{O}(\mathrm{order})$, with $\mathrm{order}$ specified below):
\begin{enumerate}
  \item[(M1)] initial guesses $\bm{u}_h^{n,0}=\mathcal{O}(1)$,
$p_h^{n,0}=\mathcal{O}(p_{ref})$,
  \item[(M2)] initial data $\bm{u}_h^0=\mathcal{O}(1)$,
$p_h^0=\mathcal{O}(p_{ref})$,
  \item[(M3)] Dirichlet boundary conditions for
$\bm{u}_h=\mathcal{O}(1)$, $p_h=\mathcal{O}(p_{ref})$,
  \item[(M4)] source terms $\bm{f}^n=\mathcal{O}(p_{ref}),\
S_f^n=\mathcal{O}(1)$.
\end{enumerate}

As before, for a fixed scenario, we consider different values for
$\delta\in[1,2.5]$ and $\kappa\in[10^{-15},10^{-10}]$. We repeat each of
the modifications~(M1)--(M4) for 20 random scenarios and take the
average of the number of iterations in the end. The results are
displayed in Figures~\ref{figure:performance-r1-r2}
and~\ref{figure:performance-r3-r4}. In order to assess the robustness of
the tuning parameter, we additionally mark the location of the
computable, optimal $\delta$ based on $K_{dr}=1.6\mu+\lambda$ and
$\beta=K_{dr}$; the identical value as in test case 1, setup 1. We
observe that despite all random variations, for all modifications, the
performance of the splitting scheme remains robust. Indeed, for each
permeability value, the optimal tuning parameter remains almost the same. Hence, our results confirm Theorem 1, independently of choice of initial and boundary data, or source terms.

\begin{figure}[h!]
\begin{subfigure}[b]{0.48\textwidth}
\includegraphics[width=\textwidth]{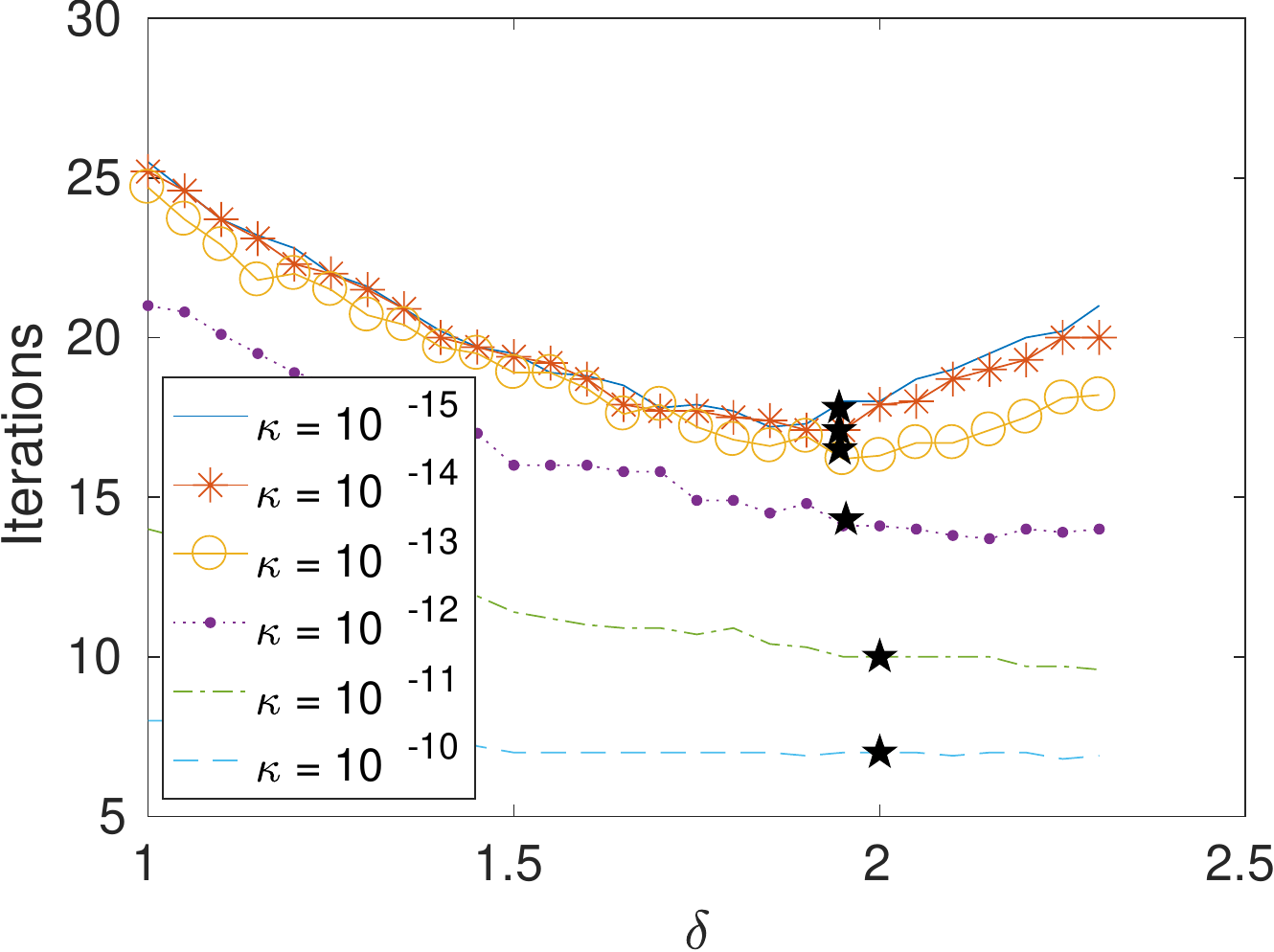}
\caption{(M1) Random initial guess}
\label{figure:performance-random-initial-guess}
\end{subfigure}
\begin{subfigure}[b]{0.48\textwidth}
\includegraphics[width=\textwidth]{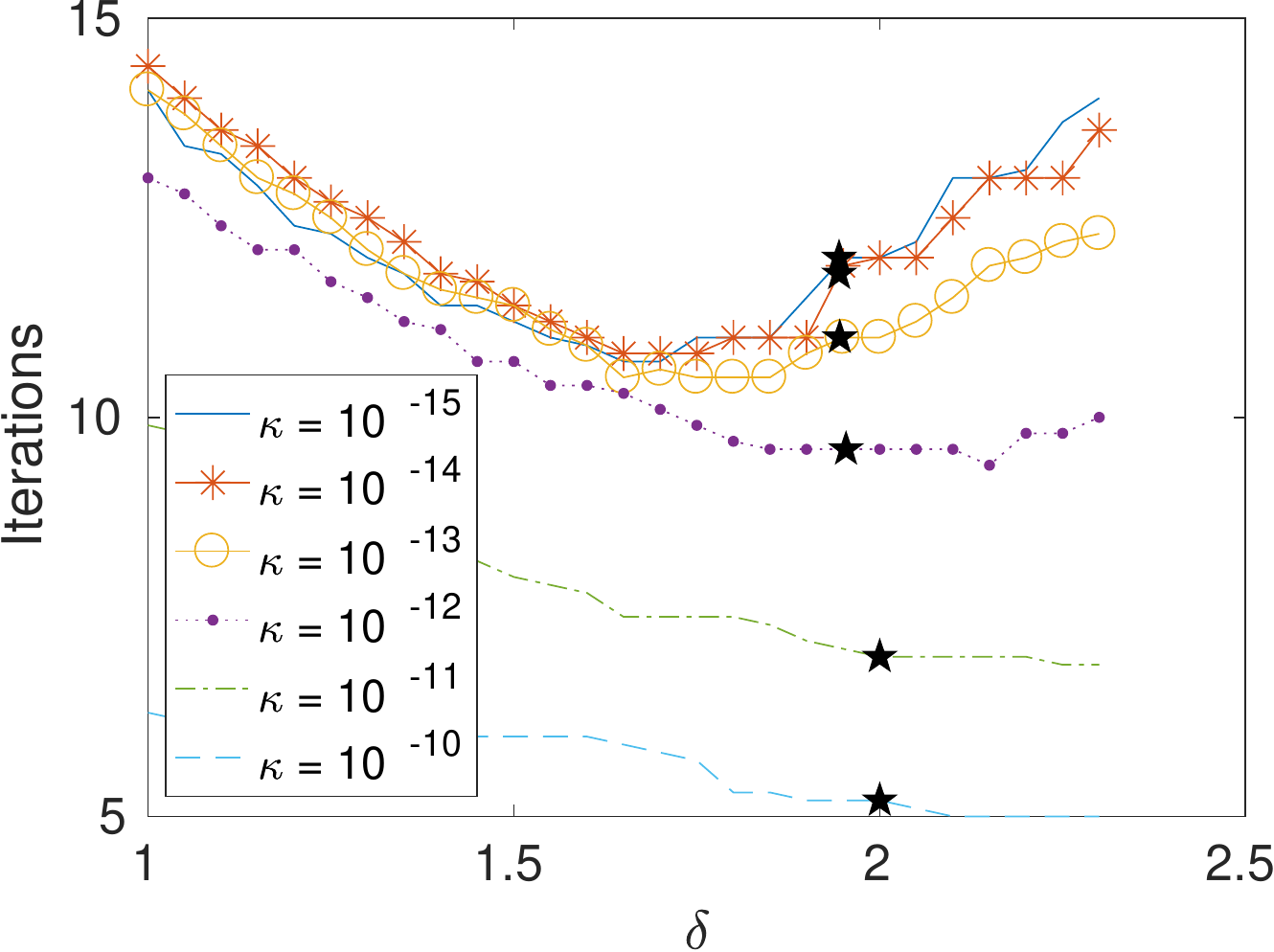}
\caption{(M2) Random initial data}
\label{figure:performance-random-initial-data}
\end{subfigure}
\caption{\label{figure:performance-r1-r2} Test case 1; random data, setup 1, P2-P1-discretization. Total number of iterations is averaged over 20 realizations. The star represents the theoretically calculated optimal $\delta$.}
\end{figure}

\begin{figure}[h!]
\begin{subfigure}[b]{0.48\textwidth}
\includegraphics[width=\textwidth]{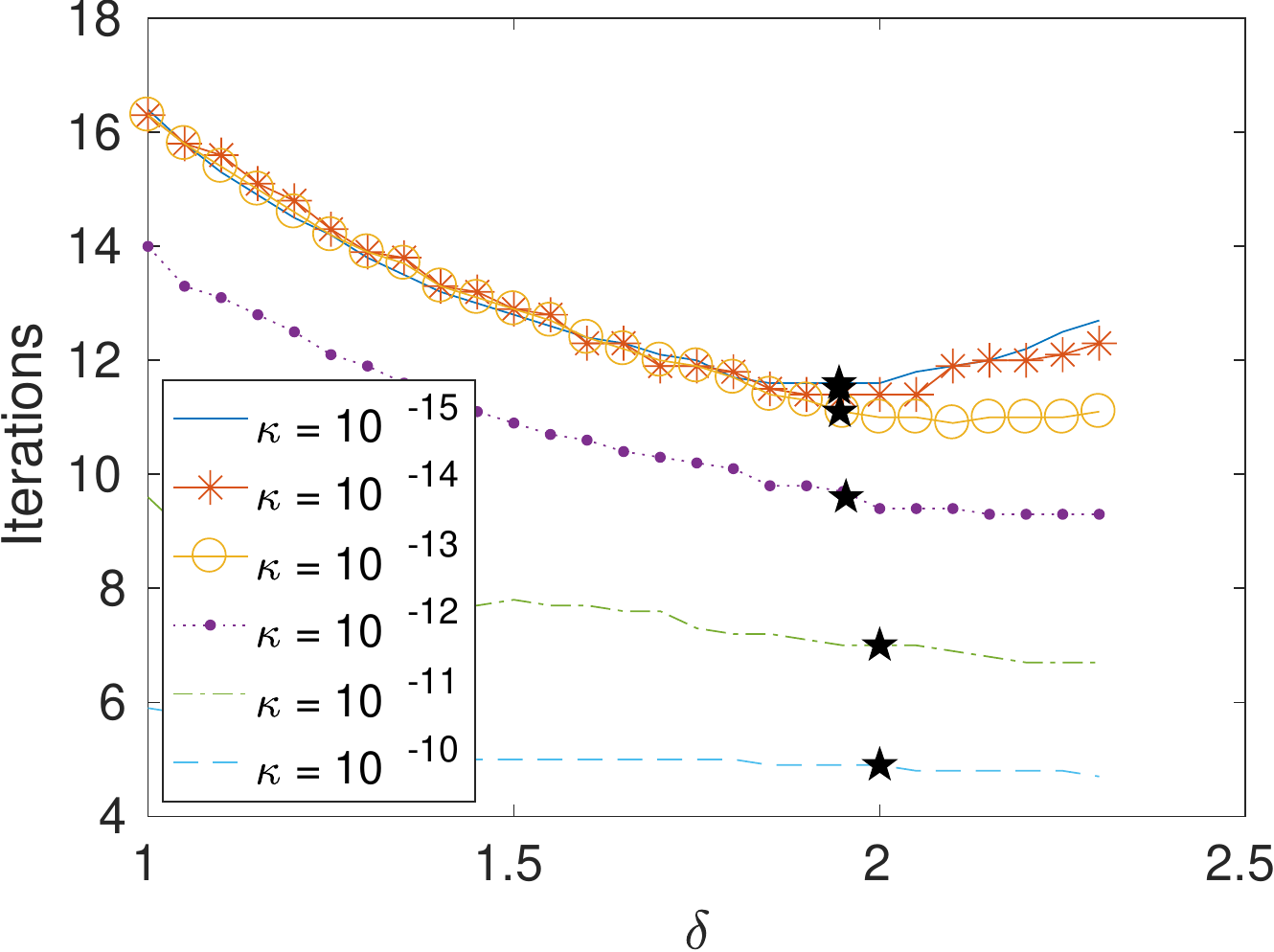}
\caption{(M3) Random Dirichlet data}
\label{figure:performance-random-dirichlet-data}
\end{subfigure}
\begin{subfigure}[b]{0.48\textwidth}
\includegraphics[width=\textwidth]{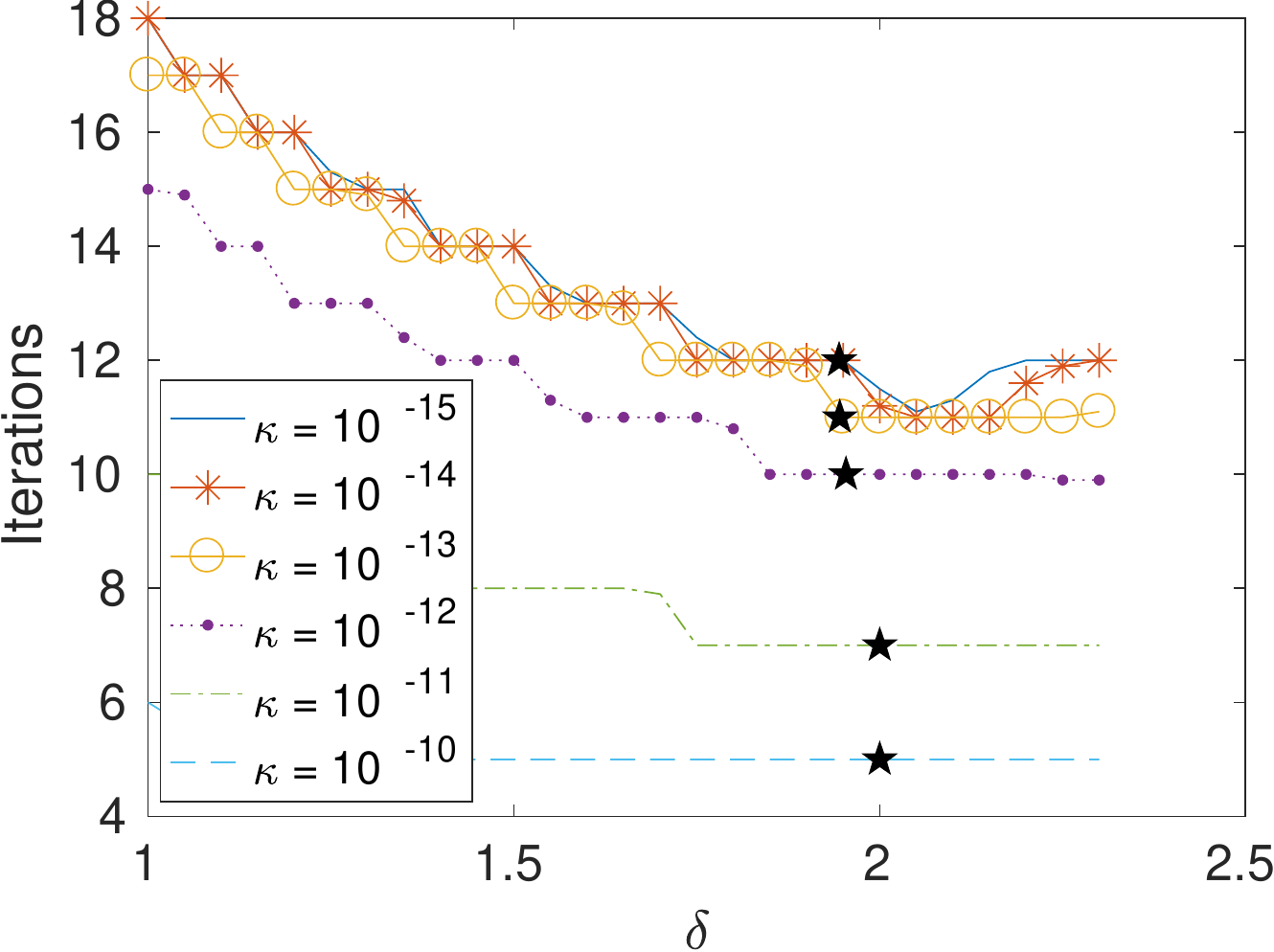}
\caption{(M4) Random source terms}
\label{figure:performance-random-source-terms}
\end{subfigure}
\caption{\label{figure:performance-r3-r4} Test case 1; random data, setup 1, P2-P1-discretization. Total number of iterations is averaged over 20 realizations. The star represents the theoretically calculated optimal $\delta$.}
\end{figure}

In the spirit of the difference of setup 1 and setup 2 in section~\ref{section:numerical-results:parabola}, we consider one more modification: A random
\begin{itemize}
 \item[(M5)] distribution of Dirichlet and Neumann boundary for $\bm{u}_h$ and $p_h$ with homogeneous data on both boundaries.
\end{itemize}
We use two different values for $K_{dr}$; (i) $K_{dr}=1.6\mu+\lambda$, as before, and (ii) $K_{dr}=\mu+\lambda$, the most pessimistic choice, suitable for two-dimensional situations.  The corresponding performance of the splitting scheme against $\delta$ is displayed in Figure~\ref{figure:performance-r5}. As expected from the observations for both setups for test case 1,
the choices $K_{dr}=1.6\mu+\lambda$, $\beta=K_{dr}$ do not yield an optimal tuning parameter using~\eqref{optimal-tuning-parameter}. However, using the worst-case choice $K_{dr}=\mu+\lambda$, $\beta=K_{dr}$, one obtains in average an acceptable match of the computable optimal values for $\delta$ and the practical optima.

\begin{figure}[h!]
\begin{subfigure}[b]{0.498\textwidth}
\includegraphics[width=\textwidth]{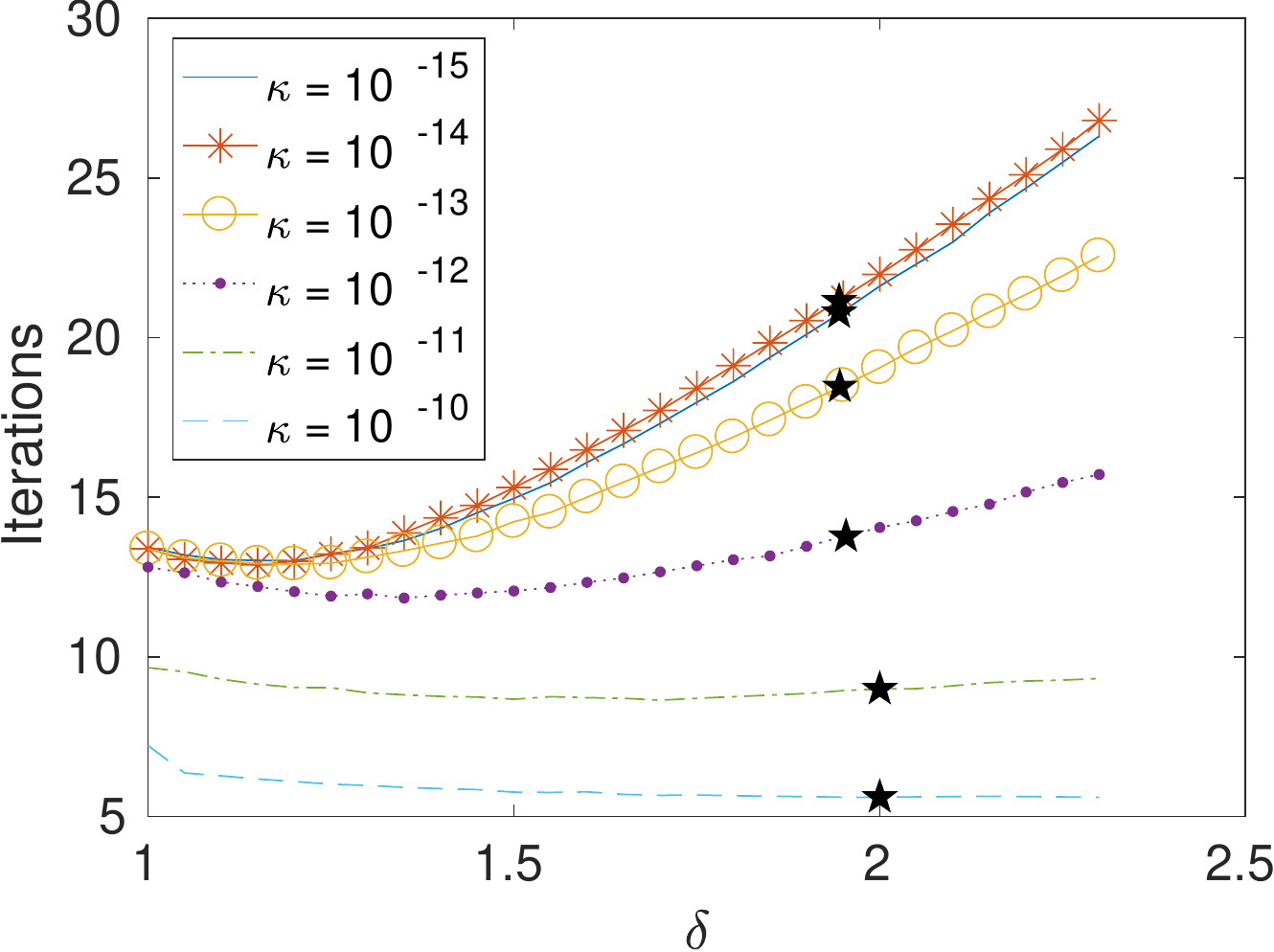}
\caption{(M5) Rand.\ Dirichlet boundary, $K_{dr}=1.6\mu+\lambda$}
\label{figure:performance-random-dirichlet-boundary-1}
\end{subfigure}
\begin{subfigure}[b]{0.48\textwidth}
\includegraphics[width=\textwidth]{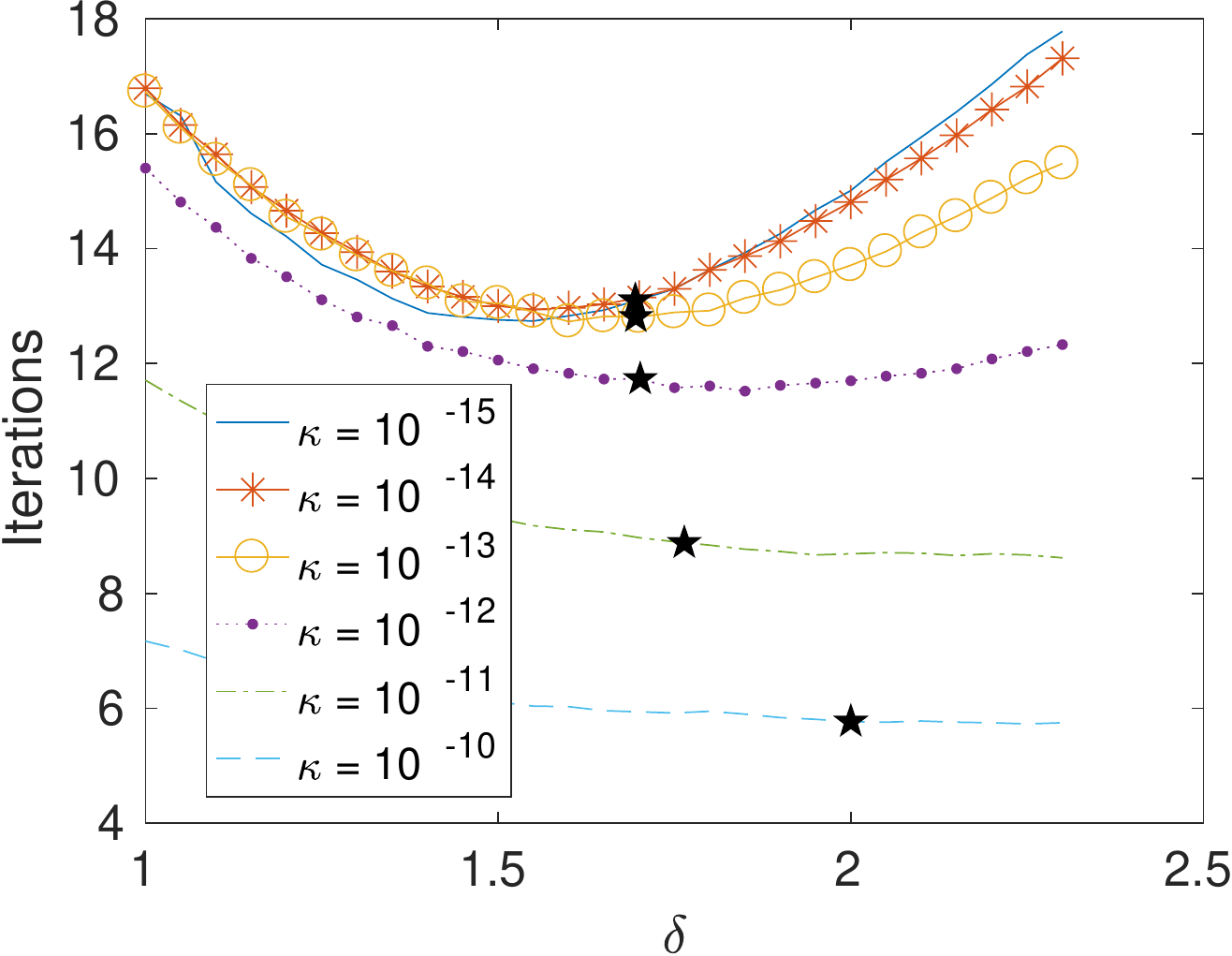}
\caption{(M5) Rand.\ Dirichlet boundary, $K_{dr}=\mu+\lambda$}
\label{figure:performance-random-dirichlet-boundary-2}
\end{subfigure}
\caption{\label{figure:performance-r5} Test case 1; random data, setup 1, P2-P1-discretization. Total number of iterations is averaged over 20 realizations. The star represents the theoretically calculated optimal $\delta$.}
\end{figure}

Finally, we conclude, that given the optimal $K_{dr}$ for a given scenario, it is suitable to use the same $K_{dr}$ for a different scenario as long as one does not change the distribution of boundary conditions. In this case, one has to either find the optimal choice or use the worst case but safe choice $K_{dr}=\tfrac{2\mu}{d}+\lambda$.

\subsection{L-shaped domain}
For this test case we consider an L-shaped domain with edges, $\Gamma_1=\left\{0\right\}\times\left[0,1\right]$, $\Gamma_2=\left[0,1\right]\times\left\{0\right\}$, $\Gamma_3=\left\{1\right\}\times\left[0,0.5\right]$, $\Gamma_4=\left[0.5,1\right]\times\left\{0.5\right\}$, $\Gamma_5=\left\{0.5\right\}\times\left[0.5,1\right]$ and $\Gamma_6=\left(0,0.5\right)\times\left\{1\right\}$.
We are considering the same source terms and apply the same parameters, spatial and temporal discretization, initial data and stopping criterion as in test case 1, see Table \ref{tab:CoefficientsCase1}. Similar to setup 2 above, for the pressure homogeneous Dirichlet boundary conditions are applied on the entire boundary, and for the displacement, homogeneous Dirichlet boundary conditions are considered everywhere except at the top, $\Gamma_{6}$. On the top, we apply zero Neumann boundary conditions in the mechanics equation, \eqref{eq:mechanics}. The solution for $\kappa = 10^{-10}$ is displayed in Figure \ref{fig:SolTest2}. 
For the computations, we set $K_{dr} = 1.4\mu+\lambda$. 
\begin{figure}[h!]
\begin{subfigure}[b]{0.5\textwidth}
\includegraphics[width=\textwidth]{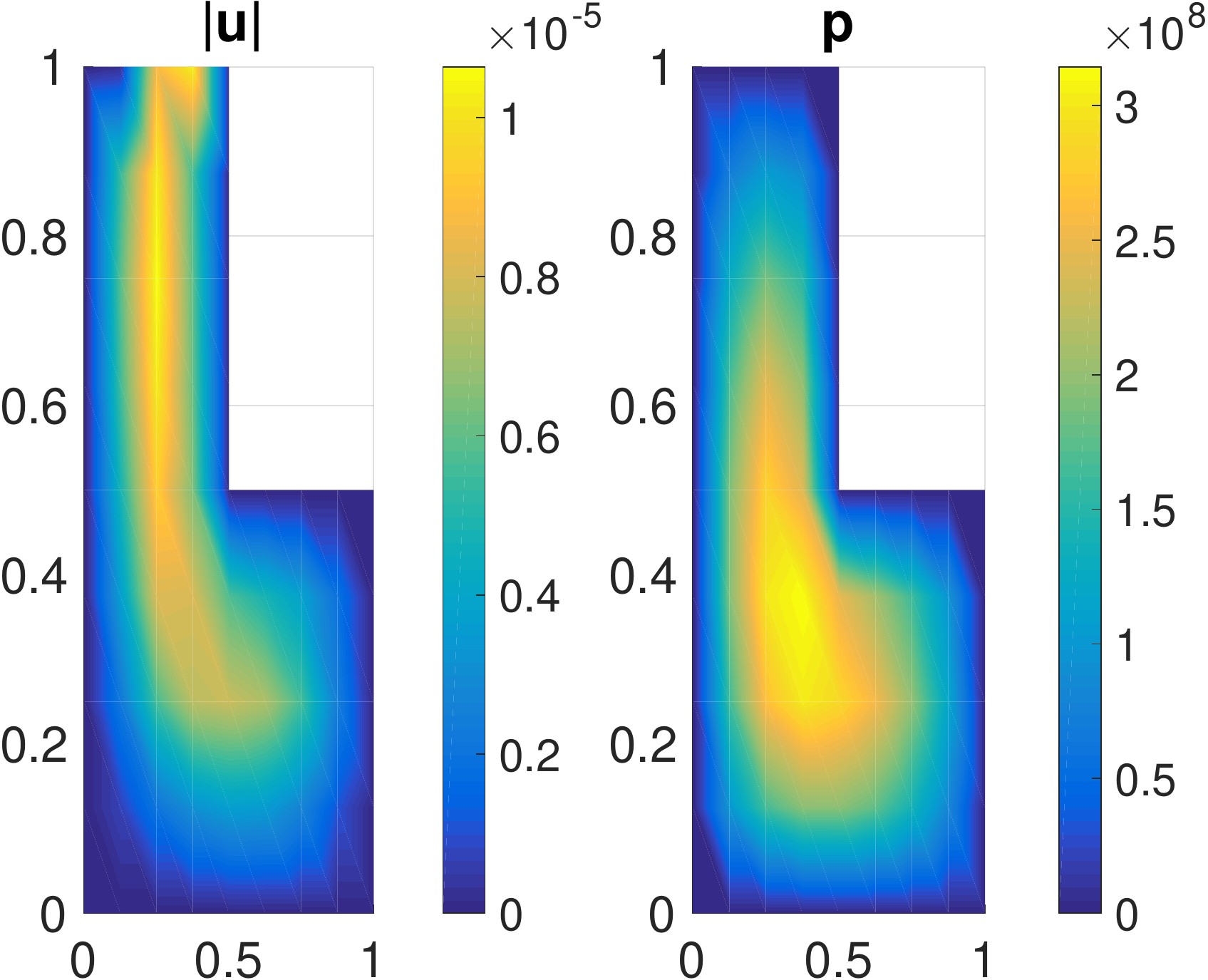}
\caption{Test case 2, L-shaped domain}
\label{fig:SolTest2}
\end{subfigure}
\begin{subfigure}[b]{0.485\textwidth}
\includegraphics[width=\textwidth]{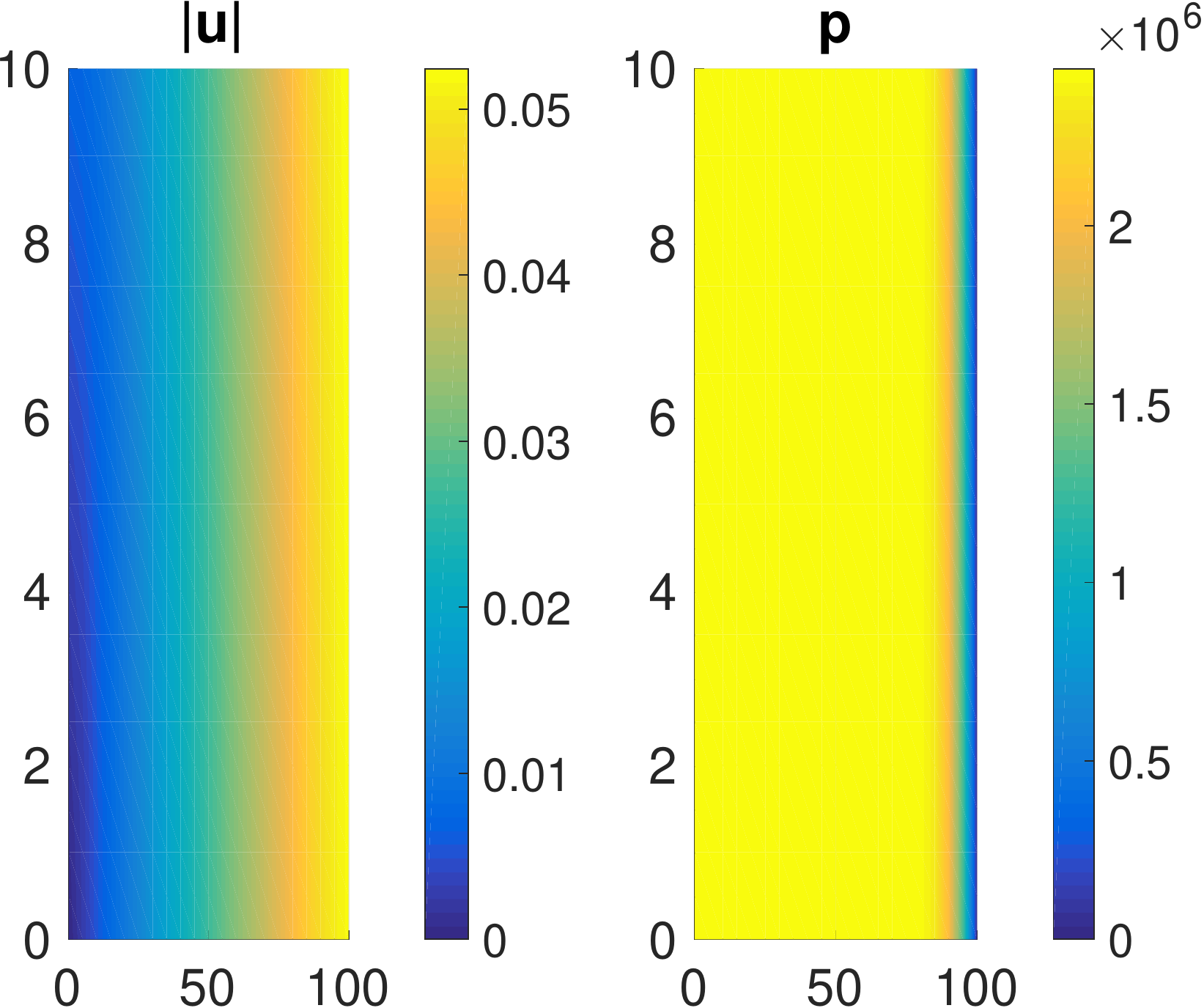}
\caption{Test case 3, Mandel's problem}
\label{fig:SolTest3}
\end{subfigure}
\label{fig:SolTestCase23}
\caption{Displacement (Left) and Pressure (Right) for test case 2 and 3. Remark that the colors are scaled differently.}
\end{figure}
\begin{figure}[h!]
\begin{subfigure}[b]{0.49\textwidth}
\includegraphics[width=\textwidth]{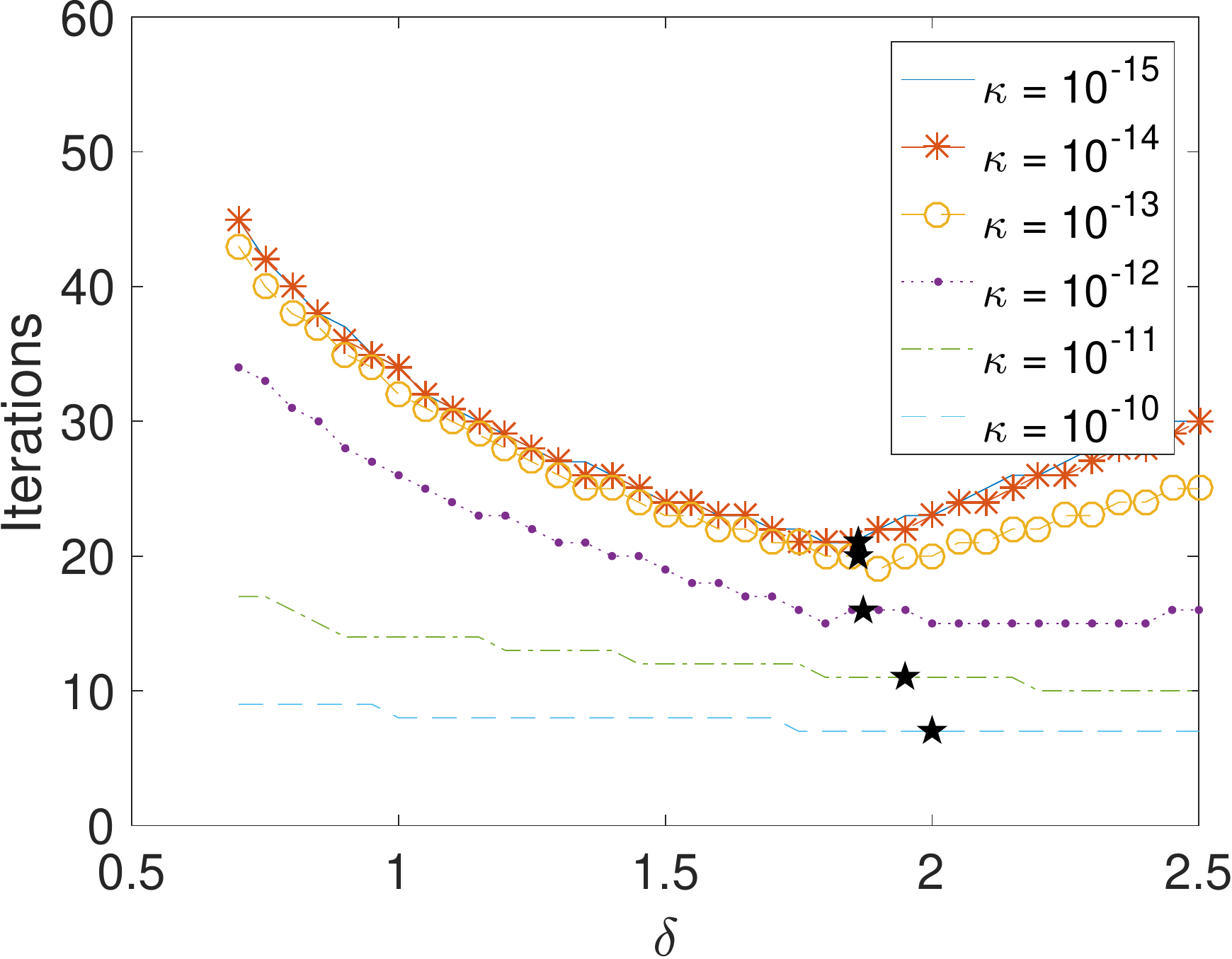}
\caption{P2-P1 discretization}
\label{fig:Test2P2}
\end{subfigure}
\begin{subfigure}[b]{0.49\textwidth}
\includegraphics[width=\textwidth]{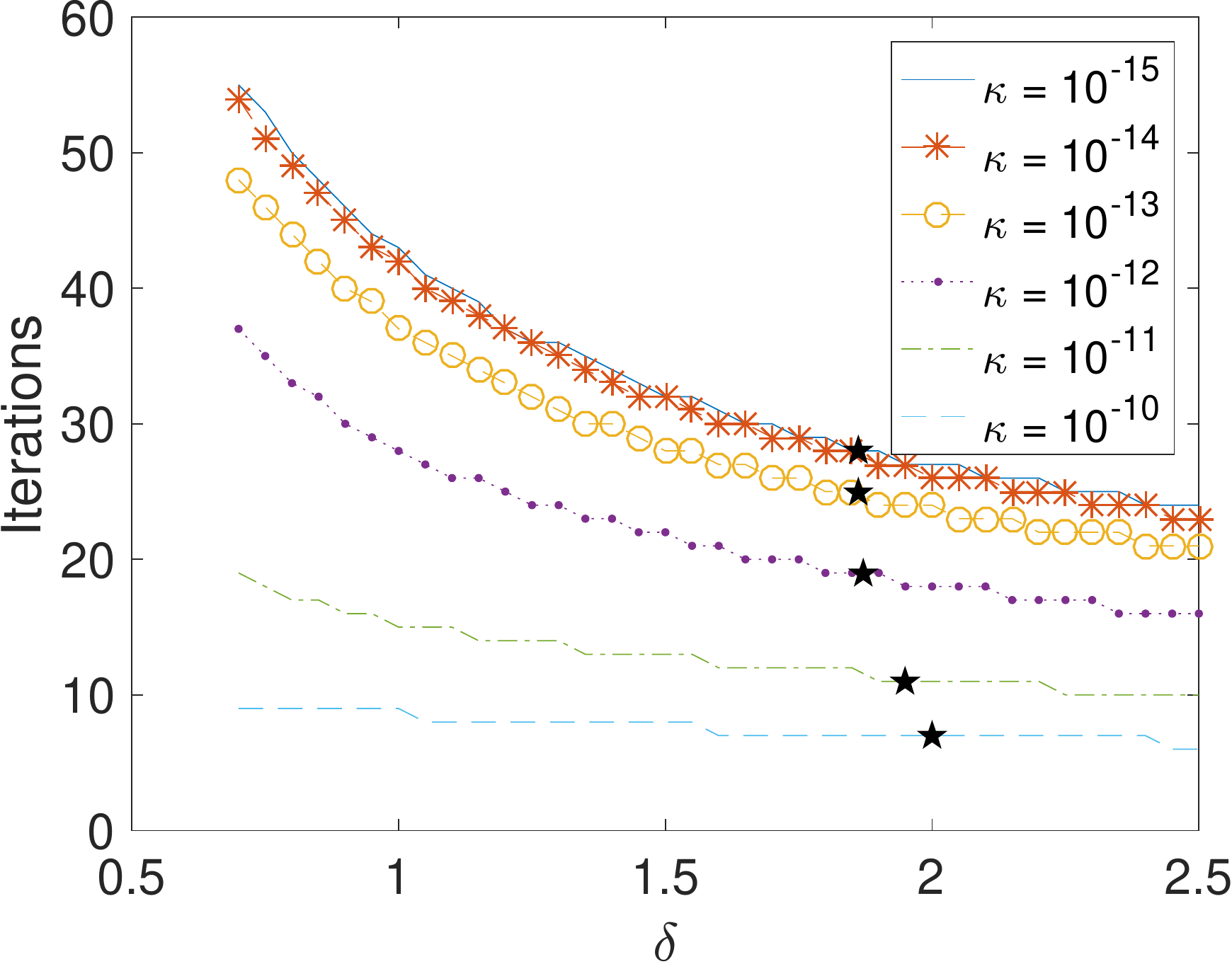}
\caption{P1-P1 discretization}
\label{fig:Test2P1}
\end{subfigure}
\label{fig:Test2}
\caption{Test case 2 (L-shaped domain): Total iteration count for one time step applying stabilization parameter $L=\frac{\alpha^2}{\delta K_{dr}}$ with $K_{dr}=1.4\mu+\lambda$. The star represents the theoretically calculated optimal $\delta$. }
\end{figure}
Again, for the stable discretization, Figure \ref{fig:Test2P2}, we observe that as the permeability increases so does the optimal choice of $\delta$. In the not inf-sup stable discretization, Figure \ref{fig:Test2P1}, however, we experience that the optimal choice lies outside the theoretical interval of $(1,2]$. It is interesting to note that the numerical results indicate that in general, the convergence for the stable P2-P1 discretization is better than the conditionally stable P1-P1, even when the latter is within a parameter regime where it is stable. 

\subsection{Mandel's Problem}
In this section we consider Mandel's problem, a relevant 2D problem with known analytical solution that is derived in \cite{Abousleiman,coussy}. The problem is often used as a benchmark problem for discretizations. The analytical expressions for pressure and displacement are given by
\begin{eqnarray}
p= &\frac{2FB \left( 1 + \nu_u \right)}{3a} \sum_{n=1}^\infty{ \frac{sin\left(\alpha_n\right)}{\alpha_n-sin\left(\alpha_n\right)cos\left(\alpha_n\right)}\left(cos\left(\frac{\alpha_n x}{a}\right)-cos\left(\alpha_n\right)\right)e^{-\frac{\alpha_n^2c_ft}{a^2}}},\label{eq:mandelanalytic1}
\\u_x=&\left[ \frac{F\nu}{2\mu a}-\frac{F\nu_u}{\mu a} \sum_{n=1}^\infty{ \frac{sin\left(\alpha_n\right)cos\left(\alpha_n\right)}{\alpha_n-sin\left(\alpha_n\right)cos\left(\alpha_n\right)}e^{-\frac{\alpha_n^2c_ft}{a^2}}}\right]x\nonumber
\\&+\frac{F}{\mu}\sum_{n=1}^\infty{\frac{cos\left(\alpha_n\right)}{\alpha_n-sin\left(\alpha_n\right)cos\left(\alpha_n\right)}sin\left(\frac{\alpha_nx}{a}\right)e^{-\frac{\alpha_n^2c_ft}{a^2}}},
\\
u_y=&\left[ \frac{-F\left(1-\nu\right)}{2\mu a}+\frac{F\left(1-\nu_u\right)}{\mu a} \sum_{n=1}^\infty{ \frac{sin\left(\alpha_n\right)cos\left(\alpha_n\right)}{\alpha_n-sin\left(\alpha_n\right)cos\left(\alpha_n\right)}e^{-\frac{\alpha_n^2c_ft}{a^2}}}\right]y, \label{eq:mandelanalytic3}
\end{eqnarray}
where $\alpha_n$, $n\in\mathbb{N}$, correspond to the positive solutions of the equation 
$$ tan\left(\alpha_n\right)=\frac{1-\nu}{\nu_u-\nu}\alpha_n,$$ and $\nu_u$, $F$, $B$, $c_f$ and $a$ are input parameters, partially depending on the physical problem parameters. Here, we employ the values listed in Table~\ref{tab:Mandel}, also used in~\cite{phillips}. For a thorough explanation of the problem and the coefficients in \eqref{eq:mandelanalytic1}--\eqref{eq:mandelanalytic3} we refer to \cite{coussy,phillips}. 

We consider the domain, $\Omega = (0,100)\times(0,10)$, discretized by a regular triangular mesh with mesh sizes $\mathrm{dx}=5$ and $\mathrm{dy}=0.5$. An equidistant partition of the time interval is applied with time step size $\tau=10$ from $t_0=0$ to $T=50$. Initial conditions are inherited from the analytic solutions, \eqref{eq:mandelanalytic1}--\eqref{eq:mandelanalytic3}. As boundary conditions, we apply exact Dirichlet boundary conditions for the normal displacement on the top, left and bottom boundary. For the pressure, we apply homogeneous boundary conditions on the right boundary. On the remaining boundaries homogeneous, natural boundary conditions are applied. The tolerances, $\epsilon_{u,r}$ and $\epsilon_{p,r}$, are set to $10^{-6}$. Our approximated solution for $\kappa=10^{-10}$ is displayed in Figure~\ref{fig:SolTest3}.

\begin{table}[h]
\begin{center}
\begin{tabular}{| l | c | r |}
\hline
Symbol & Name & Value\\
\hline
$\lambda$ & Lam\'e parameter 1 & $1.650\cdot 10^9$  \\
$\mu$ & Lam\'e parameter 2 &$2.475\cdot 10^9 $ \\
$\nu$ & Poisson's ratio & $0.2$ \\
$B$ & Skempton coefficient & $0.833$\\
$\nu_u$ & Undrained Poisson's ratio & $0.44$\\
$F$ & Applied force & $6\cdot10^{8}$ \\
$\alpha$ & Biot-Willis constant & $1$ \\
M & Compressibility coefficient & $1.650\cdot10^{10}$ \\
$c_f$ & Fluid diffusivity constant & $0.47$\\
$\kappa$ & Permeability & $10^{-14}, 10^{-13}, ..., 10^{-10}$\\
$a$ & Width of domain & $100$\\
$b$ & Height of domain & $10$\\
dx & Horizontal mesh diameter & $5$\\
dy & Vertical mesh diameter & $0.5$\\
$\tau$ & Time step size &  $10$ \\
$t_0$ & Initial time & 0\\
$T$ & Final time & $50$\\
$\epsilon_{u,r}$ and $\epsilon_{p,r}$ & Tolerances & $10^{-12}$\\
\hline
\end{tabular}
\caption{Coefficients for test case 3 (Mandel's Problem)}
\label{tab:Mandel}
\end{center}
\end{table}

\begin{figure}[h!]
\begin{subfigure}[b]{0.488\textwidth}
\includegraphics[width=\textwidth]{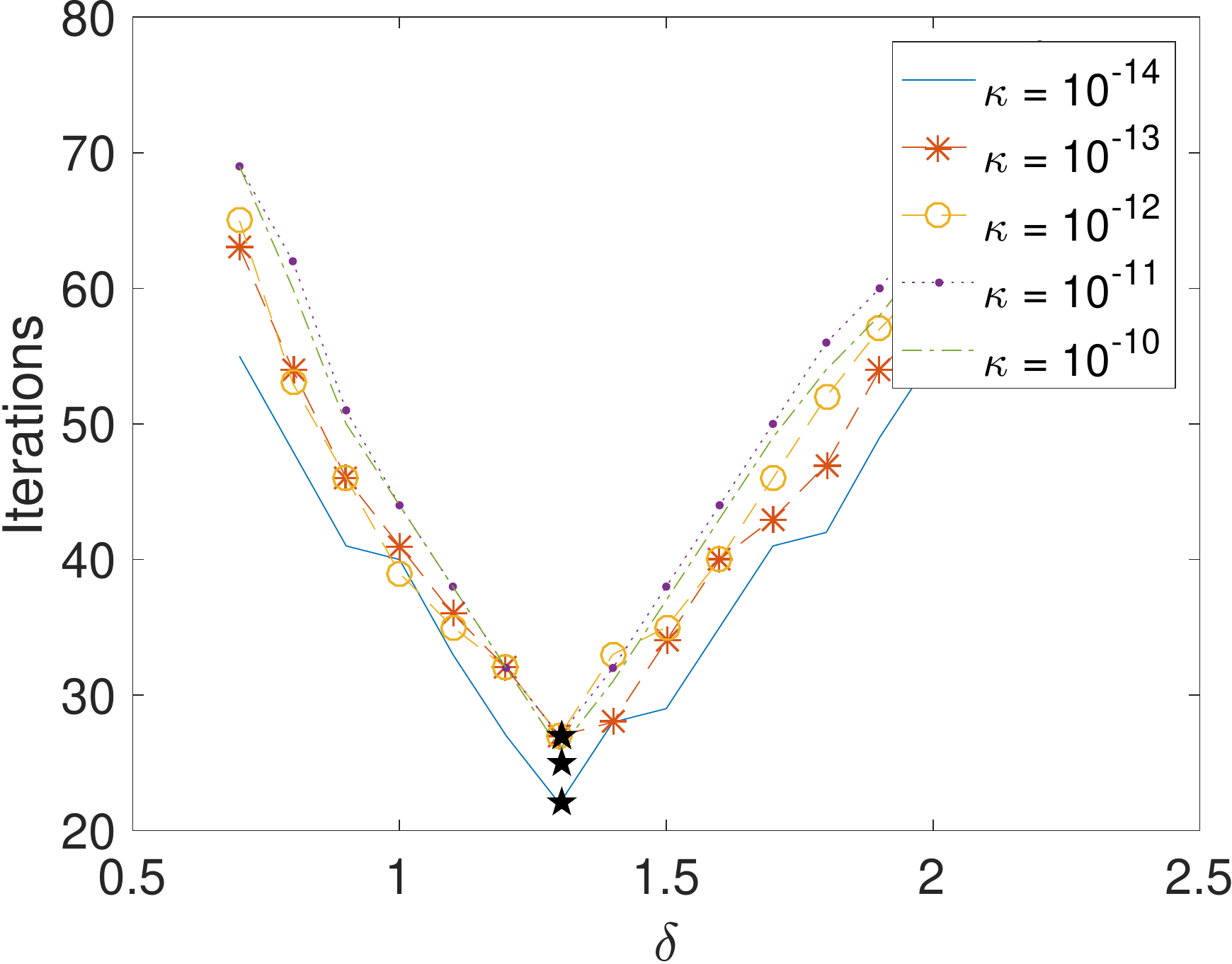}
\caption{P2-P1 discretization}
\label{fig:Test3P2}
\end{subfigure}
\begin{subfigure}[b]{0.5\textwidth}
\includegraphics[width=\textwidth]{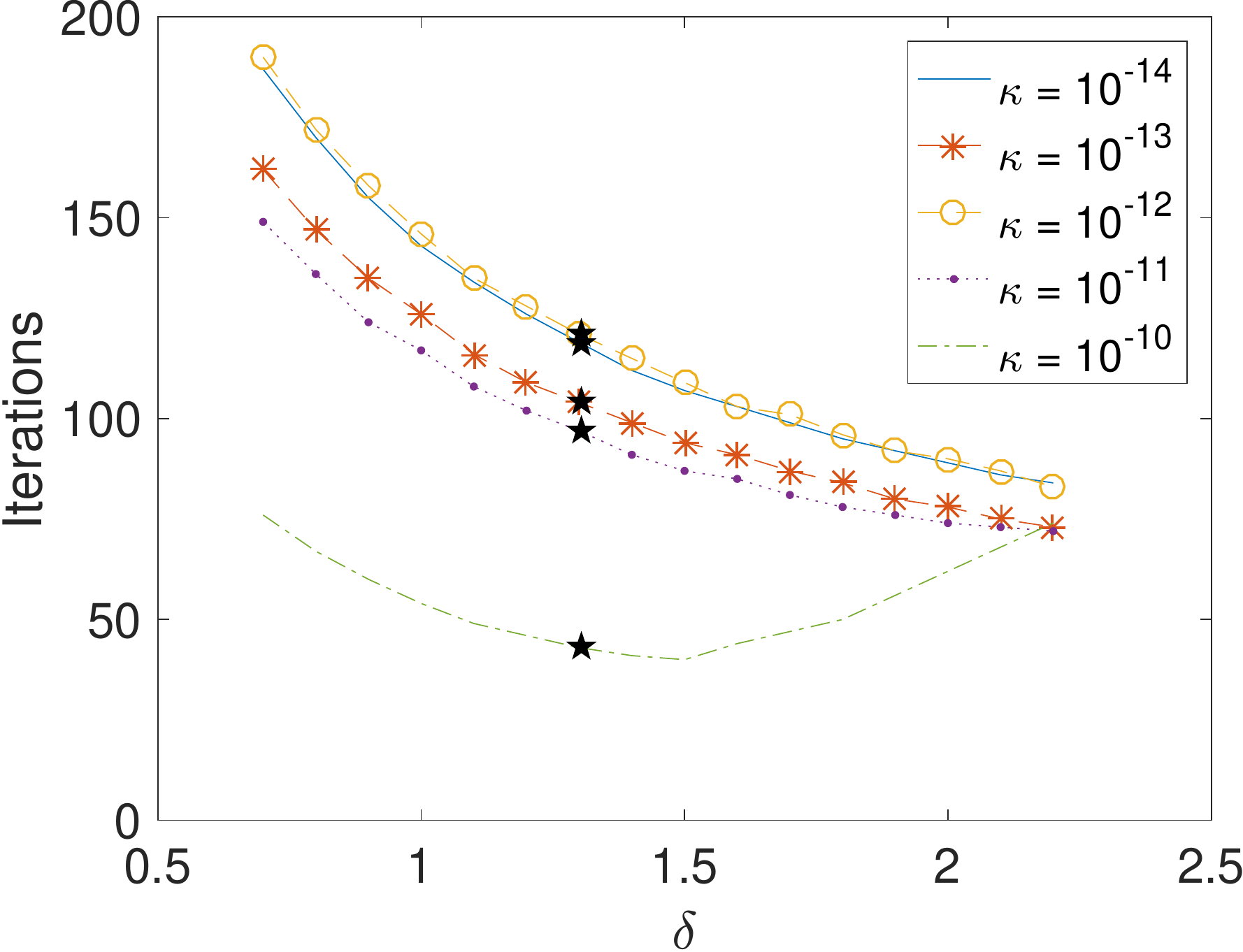}
\caption{P1-P1 discretization}
\label{fig:Test3P1}
\end{subfigure}
\label{fig:Test3}
\caption{Test case 3 (Mandel's problem), Total iteration count for five time steps applying stabilization parameter $L=\frac{\alpha^2}{\delta K_{dr}}$ with $K_{dr}=1.35\mu+\lambda$. The star represents the theoretically calculated optimal $\delta$.}
\end{figure}

Exactly as the theory predicts we observe that there is a fixed minimum for all the different permeabilities for the stable discretization, see Figure \ref{fig:Test3P2}. For the unstable discretization,  Figure \ref{fig:Test3P1}, however, we experience the same oscillatory behavior as before. There is also a clear difference in performance for the two discretizations. The inf-sup stable one performs much better, in terms of number of iterations. This is consistent with remark \ref{remark:convergence-inf-sup}.

\section{Conclusions}\label{sec:conclusions}

In this work we have considered the quasi-static, linear Biot model for poromechanics and studied theoretically and numerically the convergence of the fixed-stress splitting scheme. We have determined a formula for computing the optimal stabilization/tuning parameter, $L \in [L_{phys}/2, L_{phys}]$, depending also on the fluid flow properties and not only on the mechanics and the coupling term. We identified cases when the physical parameter $L_{phys}$ is the optimal one and cases when $L_{phys}/2$ should be taken. 

Furthermore,  we have shown for the first time that the performance of the fixed-stress scheme can be altered by a not inf-sup stable discretization. Illustrative numerical examples have been performed, including one with random initial data, random boundary conditions or random source terms, and a well-known benchmark problem, Mandel's problem. The numerical examples are in agreement with the theoretical results.


\bibliographystyle{ieeetr}

\end{document}